\documentclass [10pt,reqno]{amsart}

\newlength{\myhmargin}  
\usepackage{amsmath,amssymb,amsthm,amsfonts,amscd,flafter,
graphicx,verbatim,pinlabel,mathrsfs,enumitem}
\usepackage[all]{xy}
\usepackage{epstopdf}
\usepackage[colorlinks=false]{hyperref}
\usepackage[all]{hypcap}

\usepackage{pstricks}
\usepackage[latin1]{inputenc}

\newtheorem{theorem}{Theorem}[section]
\newtheorem{lemma}[theorem]{Lemma}

\newtheorem{prop}[theorem]{Proposition}
\newtheorem{proposition}[theorem]{Proposition}

\newtheorem{conjecture}[theorem]{Conjecture}

\newtheorem{question}[theorem]{Question}

\theoremstyle{definition}
\newtheorem{definition}[theorem]{Definition}
\newtheorem{remark}[theorem]{Remark}

\renewcommand{\epsilon}{\varepsilon}

\newcommand{\link}{{\bf Link}}
\newcommand{\diag}{{\bf Diag}}
\newcommand{\spect}{{\bf Spect}_\F}
\newcommand{\vect}{{\bf Vect}_\F}

\newcommand{\mor}{\text{Mor}}

\DeclareMathAlphabet{\mathpzc}{OT1}{pzc}{m}{it}
\usepackage{mathrsfs}

\newcommand{\adata}{\mathfrak{d}}

\newcommand{\Z}{\mathbb{Z}}

\newcommand{\Q}{\mathbb{Q}}

\newcommand{\R}{\mathbb{R}}
\newcommand{\F}{\mathbb{F}}

\newcommand{\HI}{I^\sharp}
\newcommand{\CI}{C^\sharp}

\renewcommand{\bar}{\overline}

\newcommand{\cA}{\mathcal{A}}
\newcommand{\cB}{\mathcal{B}}

\newcommand{\HF}{\widehat{\text{\em HF}}}
\newcommand{\CF}{\widehat{\text{\em CF}}}
\newcommand{\Khr}{\text{\em Khr}}
\newcommand{\Kh}{\text{\em Kh}}
\newcommand{\CKh}{\text{\em CKh}}
\newcommand{\CKhr}{\text{\em CKhr}}

\newcommand{\cKhr}{\text{\em CKhr}}
\newcommand{\cKh}{\text{\em CKh}}

\begin{document}
\thispagestyle{empty}
\title[On the functoriality of Khovanov-Floer theories]{On the functoriality of Khovanov-Floer theories}
\author[John A. Baldwin]{John A. Baldwin}
\address{Department of Mathematics \\ Boston College}
\email{john.baldwin@bc.edu}

\author[Matthew Hedden]{Matthew Hedden}
\address{Department of Mathematics \\ Michigan State University}
\email{mhedden@math.msu.edu}

\author[Andrew Lobb]{Andrew Lobb}
\address{Department of Mathematical Sciences\\ Durham University}
\email{andrew.lobb@durham.ac.uk}

\thanks{JAB was  supported by NSF Grants DMS-1104688, DMS-1406383, and NSF CAREER Grant DMS-1454865.}

\thanks{MH was partially supported by NSF CAREER Grant DMS-1150872, and an Alfred P. Sloan Research Fellowship.}

\thanks{AL was partially supported by EPSRC Grants EP/K00591X/1 and EP/M000389/1.}

\begin {abstract}
We introduce the notion of a \emph{Khovanov-Floer theory}. Roughly,  such a theory assigns a filtered chain complex over $\Z/2\Z$ to a link diagram such that (1) the $E_2$ page of the resulting spectral sequence is naturally isomorphic to  the Khovanov homology of the link; (2) this filtered complex  behaves nicely under  planar isotopy, disjoint union, and  1-handle addition; and (3) the spectral sequence collapses at the $E_2$ page for any diagram of the unlink. We prove that  a Khovanov-Floer theory naturally yields a functor from the link  cobordism category to the category of spectral sequences.  In particular, every page (after $E_1$) of the spectral sequence accompanying a Khovanov-Floer theory is a link invariant, and   an oriented  cobordism in $S^3\times [0,1]$ between links in $S^3$ induces a  map between each page of their spectral sequences, invariant up to smooth isotopy of the cobordism rel boundary. 

We  then show that the   spectral sequences relating Khovanov homology to Heegaard Floer homology and singular instanton knot homology are induced by Khovanov-Floer theories and are   therefore \emph{functorial} in the manner described above, as has been conjectured for some time. We further show that Szab{\'o}'s \emph{geometric spectral sequence} comes from a Khovanov-Floer theory, and  is thus functorial as well. In addition, we illustrate how our framework can be used to give another proof that Lee's spectral sequence is functorial and that Rasmussen's invariant is a knot invariant. Finally, we use this machinery to define some potentially new knot invariants.

\end {abstract}

\maketitle

\section{Introduction}

A primary goal of this paper is to establish the invariance and, more generally, the \emph{functoriality} of several important spectral sequences relating Khovanov homology to Floer homology.  We describe all such spectral sequences by using the general framework of a \emph{Khovanov-Floer theory}.  This framework allows us to answer, in particular, a question of Kronheimer-Mrowka from 2010 and a question of Ozsv{\'a}th-Szab{\'o} from 2003. We begin with some background and motivation.

Khovanov's groundbreaking paper \cite{kh1} associates  to a link diagram a  bigraded chain complex whose homology is, up to isomorphism,  an invariant of the underlying link type.    This invariant  \emph{categorifies} the Jones polynomial in the sense that the  graded Euler characteristic of Khovanov homology is  equal to the Jones polynomial.  One reason to promote a polynomial-valued invariant to a group-valued invariant  is that it makes sense to talk about morphisms between groups; groups form a category. This extra structure is often useful. In the case of Khovanov homology with $\F=\Z/2\Z$ coefficients, Jacobsson showed \cite{jacobsson} that a \emph{movie} for a cobordism in $S^3\times[0,1]$ with starting and ending diagrams $D_0$ and $D_1$ induces a map \[\Kh(D_0)\to\Kh(D_1),\] and that \emph{equivalent} movies define the same map  (see also  \cite{bncob,khcob,CMW}). In other words, Khovanov homology is really a functor \[\Kh:\diag\to\vect\] from the diagrammatic link cobordism category (see Subsection \ref{ssec:funct}) to the category of vector spaces over $\F$.

 Rasmussen put this additional structure to spectacular use in \cite{ras3}, combining this functoriality   with work of   Lee \cite{Lee} to define a numerical invariant of knots  which provides a lower bound on the smooth $4$-ball genus. He then used this  invariant to compute the smooth $4$-ball genera of torus knots, affirming a conjecture of Milnor   first proven by Kronheimer and Mrowka using gauge theory \cite{km8}.  When combined with work of  Freedman,  Quinn, and Rudolph \cite{FQ,rud93}, Rasmussen's proof of  Milnor's conjecture also provides the first existence result for exotic $\R^4$'s which avoids gauge theory, Floer homology, or any significant tools from analysis.

Categorification has  also played  a major role in establishing connections between quantum invariants  and  Floer homology. These now ubiquitous connections generally take the form of a spectral sequence having Khovanov homology as its $E_2$ page and  abutting to the relevant Floer-homological invariant. 
The first such connection was discovered by  Ozsv{\'a}th and Szab{\'o} in \cite{osz12}. Given a based link $L\subset S^3$ with  diagram  $D$, they defined a spectral sequence with  $E_2$ page  the reduced Khovanov homology $\Khr(D)$ of $D$,    abutting to the Heegaard Floer homology $\HF(-\Sigma(L))$ of the branched double cover of $S^3$  along  $L$ with  reversed orientation.   Similar spectral sequences in  monopole, framed instanton, and plane Floer homology  have since been discovered by Bloom, Scaduto, and Daemi,  respectively \cite{bloom,scaduto,daemi}.   Perhaps most significantly, Kronheimer and Mrowka defined in \cite{km3} a spectral sequence with $E_2$ page the Khovanov homology of $D$,  abutting to the singular instanton knot homology $\HI(\bar L)$  of the mirror of $L$. This spectral sequence played a central role in their celebrated proof that Khovanov homology detects the unknot \cite{km3}. 
In addition to their structural significance, these and  related spectral sequence have been used to:
\begin{itemize}
\item study the knot Floer homology of fibered knots \cite{lrob1, lrob2},
\item establish tightness and non-fillability of certain contact structures \cite{baldpla},
\item prove that Khovanov homology detects the unknot \cite{km3},
\item prove that Khovanov's categorification of the $n$-colored Jones polynomial detects the unknot for $n\geq 2$ \cite{griw},
\item detect the unknot with Khovanov homology of certain satellites \cite{heddcable, heddwatson},
\item prove that  Khovanov homology detects the unlink \cite{heddni, batseed}.
\item relate Khovanov homology to the twist coefficient of braids \cite{heddmark}.
\end{itemize}

Each of these spectral sequences  arises from a filtered chain complex associated with a link diagram and some additional, often   analytic, data. However, one can generally show that the $(E_i,d_i)$ page of the resulting  spectral sequence  does not depend on this additional data, up to canonical isomorphism, for $i\geq 2$. Indeed, we may think of Kronheimer and Mrowka's construction as assigning to a  planar diagram $D$ for a link $L$ a sequence \[KM(D) = \{(E_i^{KM}(D),d_i^{KM}(D))\}_{i\geq 2}\] with \[E_2^{KM}(D) = \Kh(D)\,\,\textrm{ and }\,\,E_{\infty}^{KM}(D) \cong  \HI(\bar L).\] Likewise, Ozsv{\'a}th and Szab{\'o}'s construction assigns to a  planar diagram $D$ for a based link $L$ a sequence \[OS(D) = \{(E_i^{OS}(D),d_i^{OS}(D))\}_{i\geq 2}\] with \[E_2^{OS}(D) = \Khr(D)\,\,\textrm{ and }\,\,E_{\infty}^{OS}(D) \cong \HF(-\Sigma(L)).\] Given that the $E_2$ and $E_{\infty}$ pages of   these spectral sequence are, up to isomorphism, link type invariants, a natural question is whether all intermediate pages are as well. Affirmative answers to this question were given in \cite{bald7} and \cite{km3} for the Heegaard Floer and singular instanton Floer spectral sequences, respectively. In this paper, we consider the question of invariance much more widely (that is, the invariance of all spectral sequences given by what we call \emph{Khovanov-Floer theories}).  In fact, we go much further: invariance is a consequence of \emph{functoriality} of all Khovanov-Floer theories.

For now, let us continue the discussion of functoriality in the instanton and Heegaard Floer cases.  We write $\link$ to denote the link cobordism category, whose objects are oriented links in $S^3:=\R^3\cup\{\infty\}$, and whose morphisms are  isotopy classes of  oriented, collared link cobordisms   in $S^3\times [0,1]$.  In particular, two  surfaces represent the same morphism if they differ by smooth isotopy fixing a collar neighborhood of the boundary pointwise. As explained in Subsection \ref{ssec:funct}, Khovanov homology can be made into a functor \[\Kh:\link\to\vect\] in a  natural way. Meanwhile, Kronheimer and Mrowka showed that a cobordism $S$ from $L_0$ to $L_1$ gives rise to a map on singular instanton knot homology, \[\HI({-}S):\HI(\bar L_0)\to\HI(\bar L_1),\] which is an invariant of the morphism in $\link$ represented by $S$. That is, singular instanton knot homology also defines a functor \[\HI:\link\to\vect.\] So, in essence, the $E_2$ and $E_\infty$ pages of Kronheimer and Mrowka's spectral sequence behave functorially with respect to link cobordism. It is therefore natural to ask, as Kronheimer and Mrowka did in 2010, whether their \emph{entire} spectral sequence (after the $E_1$ page) defines a functor from $\link$ to the spectral sequence category $\spect$, of which an object is a   sequence $\{(E_i,d_i)\}_{i\ge i_0}$  of chain complexes over $\F$ satisfying \[H_*(E_i,d_i)=E_{i+1},\] and a morphism is a sequence of chain maps \[\{F_i :  (E_i,d_i) \rightarrow (E'_i,d'_i)\}_{i\geq i_0}\]
satisfying $F_{i+1}=(F_i)_*$. We record their question  informally as follows.

\begin{question}{\em (Kronheimer-Mrowka \cite[Section 8.1]{km3})}
\label{ques:ifunctor}
Is the spectral sequence from Khovanov homology to singular instanton knot homology functorial?
\end{question}

One can ask a similar question of Ozsv{\'a}th and Szab{\'o}'s spectral sequence. Reduced Khovanov homology can be thought of as a functor \[\Khr:\link_{\infty}\to \vect,\] where  $\link_\infty$ is the  based link cobordism category, whose objects are oriented links in $S^3$ passing through $\infty$, and whose morphisms are  isotopy classes of oriented, collared link cobordisms in $S^3\times [0,1]$ containing the arc $\{\infty\}\times[0,1]$.  In this category, two  surfaces represent the same morphism if they differ by smooth isotopy fixing both a collar neighborhood of the boundary and this arc pointwise. Given a based link cobordism $S$ from $L_0$ to $L_1$, the branched double cover of $S^3\times[0,1]$ along $S$ is a smooth, oriented 4-dimensional cobordism $\Sigma(S)$ from $\Sigma(L_0)$ to $\Sigma(L_1)$, and therefore induces a map on Heegaard Floer homology \[\HF(-\Sigma(S)):\HF(-\Sigma(L_0))\to\HF(-\Sigma(L_1)),\footnote{This map is usually denoted by $F_{-\Sigma(S)}$.}\] which is an invariant of the morphism in $\link_\infty$ represented by $S$. That is,  the Heegaard Floer homology of branched double covers  defines a functor \[\HF(\Sigma(\cdot)):\link_\infty\to\vect\] as well. This  leads to the natural question, posed  by Ozsv{\'a}th and Szab{\'o} in 2003, as to whether their  spectral sequence defines a functor from $\link_\infty$ to $\spect$.

\begin{question}{\em (Ozsv{\'a}th-Szab{\'o} \cite[Section 1.1]{osz12})}
\label{ques:hffunctor}
Is the spectral sequence from Khovanov homology to the Heegaard Floer homology of the branched double cover functorial?
\end{question}


In this paper, we answer both  Questions \ref{ques:ifunctor} and \ref{ques:hffunctor} in the affirmative. Indeed, we prove that Kronheimer-Mrowka's and Ozsv{\'a}th-Szab{\'o}'s spectral sequences are functorial, expressed more precisely in the    two theorems below. In these theorems, 
\[SV_j:\spect\to\vect\] is the forgetful functor  which sends  $\{(E_i,d_i)\}_{i\geq i_0}$ to its $j$th page $E_j$.

\begin{theorem}
\label{thm:sskm} There exists a functor \[KM:\link\to\spect\] with  $\Kh=SV_2\circ KM$ such that $KM(L) \cong KM(D)$ for any  diagram $D$ for $L$.
\end{theorem}

\begin{theorem}
\label{thm:ssos} There exists a functor \[OS:\link_\infty\to\spect\] with $\Khr=SV_2\circ OS$ such that $OS(L) \cong OS(D)$ for any  diagram $D$ for $L$.
\end{theorem}

\noindent That is, proper  isotopy classes of link cobordisms induce  well-defined maps on the intermediate pages of these spectral sequences, which agree at $E_2$ with the induced maps on Khovanov and reduced Khovanov homology. In short, each intermediate page is a \emph{functorial} link invariant.


One notable consequence of these theorems is that link isotopies determine isomorphisms of these spectral sequences. In particular, an isotopy $\phi$ taking $L$ to $L'$ determines a cylindrical cobordism $S_{\phi}\subset S^3\times[0,1]$ from $L$ to $L'$, and, therefore, a morphism \[\Psi_{\phi}:=KM(S_{\phi}):KM(L)\to KM(L')\] (likewise for based isotopies and $OS$). While interesting in its own right, this new structure also  recovers  the results from \cite{bald7} and \cite{km3} that the isomorphism classes of the intermediate
pages of these spectral sequences are  link type invariants:   the morphism $\Psi_{\phi}$ is an isomorphism in $\spect$ since the cobordism $S_{\phi}$ is an isomorphism in $\link$.

Theorems \ref{thm:sskm} and \ref{thm:ssos} follow from a much more general and powerful framework developed in this paper.  The key idea is the notion of a \emph{Khovanov-Floer theory}, alluded to above and formally introduced in Section \ref{sec:defs}. Very roughly, this is something which assigns a filtered chain complex  to a link diagram (and possibly extra data) such that (1) the $E_2$ page of the resulting spectral sequence is naturally isomorphic to  the Khovanov homology of the diagram;  (2)  the filtered complex behaves in certain nice ways under planar isotopy,  disjoint union, and  diagrammatic 1-handle addition; and (3) the spectral sequence collapses at the $E_2$ page for any diagram of the unlink. The import of this notion is indicated by our main theorem below, which asserts that the spectral sequence associated with a Khovanov-Floer theory is automatically functorial.
\begin{theorem}
\label{thm:khfunct}
The spectral sequence associated with a Khovanov-Floer theory  defines a  functor \[F:\link\to\spect\] with $\Kh = SV_2\circ F$.
\end{theorem}

In particular, the  spectral sequence defined by a Khovanov-Floer theory is, up to isomorphism, a link type invariant. What is striking is how this invariance and the additional functoriality promised in the theorem are guaranteed by just the few, rather weak conditions that go into the definition of a Khovanov-Floer theory.

To prove Theorem \ref{thm:khfunct}, we first  show that the spectral sequence associated with a Khovanov-Floer theory  defines a functor from  $\diag$ to $\spect$. The morphism of spectral sequences this functor assigns to a movie is induced by a filtered chain map between the filtered  complexes associated with the diagrams  at either end of the movie. To define this filtered chain  map, we represent the movie as a composition of \emph{elementary} movies, each corresponding to  a planar isotopy, diagrammatic handle attachment, or Reidemeister move. We  assign a filtered  map to each   elementary movie so that the induced map on $E_2$ agrees with the corresponding Khovanov map, and we define the map associated with the original movie to be the composite of these elementary movie maps. For     planar isotopy and handle attachment, these elementary maps are essentially built into the definition of a Khovanov-Floer theory. More  interesting is our assignment of filtered maps to Reidemeister moves. The idea  is to first arrange via movie moves that  the Reidemeister move  takes place amongst unknotted components. Then  one constructs the desired    map using the behavior of a Khovanov-Floer theory under disjoint union and handle attachment, and the fact that the associated spectral sequence collapses at $E_2$ for any diagram of an unlink. The fact that equivalent movies are assigned equal morphisms (so that we actually get a functor from $\diag$) follows immediately  from the fact that these morphisms agree on $E_2$ with the corresponding Khovanov map. Finally, we promote this  to a functor from $\link$ in a relatively standard way.

The power of our framework lies in the fact it is often easy to determine whether a given  construction satisfies the conditions of a Khovanov-Floer theory, whereas proving  the functoriality (or even invariance) of a construction without the benefit of this notion has proven tricky in practice (in particular, it was not known at all before this paper whether Kronheimer-Mrowka's and Ozsv{\'a}th-Szab{\'o}'s constructions were functorial). This principle is elaborated  in Remark \ref{rmk:weakstrong}. 

In Section \ref{sec:kfthys}, we show that several well-known constructions are indeed Khovanov-Floer theories.  Importantly, we prove the following.

\begin{theorem}
\label{thm:kmoskf}
Kronheimer-Mrowka's and Ozsv{\'a}th-Szab{\'o}'s spectral sequences come from Khovanov-Floer theories.\footnote{Really, Ozsv{\'a}th and Szab{\'o}'s construction is what we term a \emph{reduced} Khovanov-Floer theory.} 
\end{theorem}

Observe that Theorems \ref{thm:sskm} and \ref{thm:ssos}  follow immediately from Theorem \ref{thm:kmoskf} combined with Theorem \ref{thm:khfunct}. Though we do not do so here, one can   show that the spectral sequences defined by Bloom, Scaduto, and Daemi  also come from Khovanov-Floer theories and are therefore functorial as well.

The other examples in  Section \ref{sec:kfthys} concern  constructions which do \emph{not}  come from Floer homology. The first of these is  Szab{\'o}'s \emph{geometric spectral sequence}  \cite{szabo}, which relates the Khovanov homology of $L$ to another combinatorial link type invariant which (though defined without  Floer homology) is conjecturally isomorphic to \[\HF(-\Sigma(L))\oplus\HF(-\Sigma(L)).\]   We prove the following.

\begin{theorem}
\label{thm:szabokf}
Szab{\'o}'s spectral sequence comes from a Khovanov-Floer theory.
\end{theorem}

\noindent  Theorem \ref{thm:szabokf} provides an easy, alternative proof of Szab{\'o}'s result that this spectral sequence is a link type invariant, while furthermore showing that it behaves functorially with respect to link cobordism.

For another example, we consider Lee's deformation of Khovanov homology  \cite{Lee}. For knots, this deformation produces a  spectral sequence 
abutting to the direct sum $\F\oplus\F$, with each summand  supported in a single quantum grading.   Rasmussen's invariant, mentioned earlier, may be described as the average of these  two gradings. We can easily prove the following.

\begin{theorem}
\label{thm:leekf}
Lee's spectral sequence comes from a Khovanov-Floer theory.\footnote{We actually prove this for a version of Lee's spectral sequence defined  over $\F$ by Bar-Natan  \cite{bncob}.}
\end{theorem}

\noindent This yields an easy, alternative proof that Lee's spectral sequence is a link type invariant, from which it follows that Rasmussen's invariant $s_\F$ is as well.

Apart from their theoretical appeal, we expect our functoriality results to have applications to  computing  Floer theories and the maps on Floer homology induced by link cobordisms. 
Indeed, in the singular instanton and Heegaard Floer settings, one can show that the morphism of spectral sequences we assign to a cobordism is induced by a filtered chain map whose  induced map on total homology agrees with the cobordism map on Floer homology.  In the case of Kronheimer and Mrowka's construction, for example, this means that there is a commutative diagram
\begin{equation*}
\begin{gathered}
 \xymatrix@C=30pt@R=25pt{
H_*(C(D_0)) \ar[d]_{\cong} \ar[r]^-{(f_M)_*}  & H_*(C(D_1)) \ar[d]^{\cong} \\
I^\sharp(\bar L_0) \ar[r]_-{I^\sharp(-S)} & I^\sharp(\bar L_1).
}
\end{gathered}
\end{equation*}
Here, $C(D_i)$ is the  filtered complex associated to a diagram $D_i$ for a link $L_i$ which gives rise to Kronheimer and Mrowka's spectral sequence, and $f_M$ is the  filtered  chain map  associated to a movie  $M$ for the cobordism $S$ which induces the morphism of spectral sequences \[KM(S):KM(L_0)\to KM(L_1)\] in Theorem \ref{thm:sskm}.  
The third author and Zentner \cite{lobbzentner}  recently used the idea that diagrammatic 1-handle additions induce morphisms of spectral sequences to compute the  singular instanton and Heegaard Floer spectral sequences for a variety of knots, even without the assumption (proved in this paper) that the morphism associated to a movie for a cobordism is independent of the movie. The functoriality established here should allow us to extend these sorts of calculations to a wider array of knots.

Another concrete and important application of our framework has to do with  proving  topological invariance for  Floer-homological constructions. For example, Herald, Kirk and the second author  recently defined a Lagrangian Floer analogue of singular instanton knot homology, which they call \emph{pillowcase Floer homology}.  They do not know how to give a direct proof that their construction defines a knot invariant. They plan to bypass this difficulty by showing that pillowcase Floer homology is isomorphic to the $E_\infty$ page of the spectral sequence associated with a Khovanov-Floer theory, from which invariance will follow automatically.

In a slightly different direction, the results in this paper imply that any reasonably well-behaved deformation of the Khovanov chain complex gives rise to  link and cobordism invariants. That is, our construction gives a mechanism for constructing a wealth of new  invariants. To illustrate this principle,  we actually construct in Subsection \ref{subsec:new_knot_invariants} some new deformations of the Khovanov complex which are easily shown to define Khovanov-Floer theories.  At the moment, however, we do not know whether the resulting link and  cobordism invariants are  different from those in Khovanov homology.  A natural  (and probably very difficult) problem is to classify the link  invariants that  come from Khovanov-Floer theories.  


\subsection{Organization}

 In Section \ref{sec:bkgnd}, we collect some facts from homological algebra and review Khovanov homology and ideas involving functoriality. In Section \ref{sec:defs}, we give a precise definition of a Khovanov-Floer theory. In Section \ref{sec:proofs}, we prove our main result, Theorem \ref{thm:khfunct}. In Section \ref{sec:kfthys}, we show that the spectral sequence constructions of Kronheimer-Mrowka, Ozsv{\'a}th-Szab{\'o}, Szab{\'o}, and Lee constitute Khovanov-Floer theories, and we describe some new deformations of the Khovanov complex which also define Khovanov-Floer theories.

\subsection{Acknowledgements} It is our pleasure to thank Scott Carter and Ciprian Manolescu for helpful conversations.

\section{Background}
\label{sec:bkgnd}
\noindent We will work over $\F=\Z/2\Z$ throughout the entire paper unless otherwise specified.

\subsection{Homological algebra}
\label{ssec:homalg}In this subsection, we  record some  basic results about filtered chain complexes and their associated spectral sequences.

The filtered chain complexes considered in this paper are all   chain complexes over $\F=\Z/2\Z$, admitting a direct sum decomposition of the form \begin{equation}\label{eqn:filtcpx}(C=\bigoplus_{i\geq i_0}C^i,\,d=d^{0}+d^{1} + \dots),\end{equation}
where: 
\begin{itemize}
\item $d^i(C^j)\subset C^{j+i}$ for each $j\geq i_0$, and 
\item $C^i=\{0\}$ for all $i$ greater than some $i_1$.
\end{itemize}  We consider elements of $C^i$ to be homogeneous of \emph{grading} $i$.  This grading should not be confused with a (co)homological grading (i.e. a grading raised by one by $d$) which, while generally present, will be suppressed throughout the discussion.  The associated filtration  
\begin{equation}\label{eqn:filti1}C=\mathscr{F}^{i_0}\supset \mathscr{F}^{i_0+1}\supset \dots \supset \mathscr{F}^{i_1}=\{0\}\end{equation} is  given by \[\mathscr{F}^{i} = \bigoplus_{j\geq i}C^j.\] In fact, \emph{every} filtered complex over $\F$ (or any other field)  can be thought of in terms of a graded complex in which the differential does not decrease grading, as above. From this perspective, a \emph{filtered chain map of degree $k$} from $(C,d)$ to $(C',d')$ is a chain map $f:C\to C'$ admitting a splitting \begin{equation}\label{eqn:filtmap}f = f^k+f^{k+1} + f^{k+2}+  \dots \end{equation} such that $f^i(C^j)\subset (C')^{j+i}$. 

A {\em spectral sequence} is a sequence of chain complexes $\{(E_i,d_i)\}_{i\geq i_0}$ for some $i_0\geq 0$ satisfying \[E_{i+1}=H_*(E_i,d_i).\] A filtered complex $(C,d)$  gives rise to a spectral sequence \[\{(E_i(C),d_i(C))\}_{i\geq 0}\] of graded vector spaces via the standard \emph{exact couple} construction; see, e.g.  \cite[Section 14]{BottTu}. Note that each $E_i(C)$ inherits a grading from that of $C$. As usual, we  will write $E_i(C)=E_{\infty}(C)$ to mean that \[E_i(C) = E_{i+1}(C)=E_{i+2}(C)=\dots:=E_{\infty}(C).\] A \emph{morphism} from a spectral sequence $\{(E_i,d_i)\}_{i\geq i_0}$ to a spectral sequence  $\{(E'_i,d'_i)\}_{i\geq i'_0}$ is a sequence of chain maps   \[\{F_i:(E_i,d_i)\to (E'_i,d'_i)\}_{i\geq \max\{i_0,i'_0\}}\] satisfying $F_{i+1}=(F_i)_*$.
A   filtered chain map as in (\ref{eqn:filtmap}) gives rise to a morphism of spectral sequences   \[\{F_i=E_i(f):(E_i(C),d_i(C))\to (E_i(C'),d_i(C'))\}_{i\geq 0}\]  in a standard way as well.  If the filtered map is of degree $k$, then each map in the morphism is  homogenous of degree $k$ with respect to the grading. As mentioned in the introduction, spectral sequences and their morphisms form a category which we denote by $\spect$.

The  three lemmas below are the main results of  this subsection; we will make heavy use of them in Sections \ref{sec:defs} and \ref{sec:proofs}. 

\begin{lemma}
\label{lem:ssmapiso}
Suppose \[f:(C,d)\to (C',d')\] is a degree $0$ filtered chain map  such that $E_i(f)$ is an  isomorphism. Then $E_j(f)$ is an isomorphism for all $j\geq i$. Moreover, there exists a degree $0$ filtered chain map \[g:(C',d')\to (C,d)\]  such that $E_j(g) = E_j(f)^{-1}$ for all $j\geq i$.
\end{lemma}

\begin{lemma}
\label{lem:sseqmaps}
Suppose \[f,g:(C,d)\to (C',d')\] are degree $k$  filtered chain  maps such that $E_i(f) = E_i(g)$. Then $E_j(f) = E_j(g)$ for all $j\geq i$.
\end{lemma}

\begin{lemma}
\label{lem:einfty}
Suppose $E_i(C) = E_{\infty}(C)$. Then there exists a degree $0$ filtered chain map \[f:(C,d)\to (E_i(C),0)\] from $(C,d)$ to the complex consisting of the  vector space $E_i(C)$ with trivial differential such that the induced map \[E_i(f):E_i(C)\to E_i(C)\] is the identity map.
\end{lemma}

The remainder of this section is devoted to proving these lemmas (even though they  are well-known to experts).
We  will do so using a procedure called \emph{cancellation} which provides  a   concrete way of understanding these spectral sequences and the maps between them. We  first describe this procedure for ordinary (unfiltered) chain complexes, as part of the well-known \emph{cancellation lemma} below.

\begin{lemma}[Cancellation Lemma]
\label{lem:cancel}
Suppose  $(C,d)$ is a chain complex over $\F$ freely generated by elements $\{x_i\}$ and let $d(x_i,x_j)$ be the coefficient of $x_j$ in $d(x_i)$. If $d(x_k,x_l)=1,$ then the complex $(C',d')$ with generators $\{x_i| i \neq k,l\}$ and differential $$d'(x_i) = d(x_i) + d(x_i,x_l)d(x_k)$$ is chain homotopy equivalent to $(C,d)$ via  the chain homotopy equivalences \[\pi:C\rightarrow C'\,\,\,\text{ and }\,\,\,\iota: C'\rightarrow C\]  given by \[\pi = P\circ(id + d\circ h)\,\,\,\text{ and }\,\,\,\iota = (id+h\circ d)\circ I,\] where $P$ and $I$ are the  natural projection and inclusion maps and $h$ is the linear map defined by 
\[
h(x_l)=x_k\,\,\,\text{ and }\,\,\,
h(x_i)=0\text{ for }i\neq l.
\]
We say that the complex $(C',d')$  is obtained from $(C,d)$ by \emph{canceling} the component of  $d$ from $x_k$ to $x_l$.
\end{lemma}

 \begin{remark}
 \label{rmk:cancelhomology}
 The homology $H_*(C,d)$ of the complex in Lemma \ref{lem:cancel} can be understood as the vector space obtained by  performing cancellation until the resulting differential is zero. Technically, the actual vector space resulting from this cancellation depends on the order of cancellations, but any such vector space is canonically isomorphic to $H_*(C,d)$.
 \end{remark}

Suppose now that $(C,d)$ is a filtered chain complex as in (\ref{eqn:filtcpx}). One may think of the   sequence $\{E_i(C)\}_{i\geq 0}$ as the sequence of graded vector spaces obtained by performing cancellation in  stages, where  the $i$th \emph{page}  records the result of this cancellation after the $i$th stage. Specifically, let:
\begin{itemize}
\item $(C_{(0)},d_{(0)}) = (C,d)$, and inductively let
\item $(C_{(i)},d_{(i)})$ be the complex obtained from $(C_{(i-1)},d_{(i-1)})$ by canceling the components of $d_{(i-1)}$ which shift the grading by $i-1$.
\end{itemize} Then $E_i(C)$ may be thought of as   the graded vector space $C_{(i)}$, with grading naturally inherited from $C$. Under this formulation, the  spectral sequence differential $d_k(C)$ on $E_k(C)$ is the sum of the components of $d_{(k)}$ which shift the grading by exactly $k$, so that the recursive condition above  may be interpreted as the more familiar  \[E_i(C) = H_*(E_{i-1}(C),d_{i-1}(C)),\] per Remark \ref{rmk:cancelhomology}.

Suppose that $f$ is a filtered chain map of degree $k$  as in (\ref{eqn:filtmap}). Cancellation  provides a nice way of understanding the induced maps 
\[E_i(f):E_i(C)\to E_i(C')\] for each $i\geq 0$. Specifically, every time we cancel a component of $d$ or $d'$, we may adjust the components of $f$ as though they were components of a differential (they \emph{are} components of the mapping cone differential). In this way, we obtain an \emph{adjusted map} \[f_{(i)}:(C_{(i)}, d_{(i)})\to (C'_{(i)}, d'_{(i)})\] for each $i\geq 0$. The induced map $E_i(f)$  may then be understood as  the sum of the components of $f_{(i)}$ which shift the grading by exactly $k$. Note that if \[f:(C,d)\to(C',d')\,\,\,\text{ and }\,\,\,g:(C',d')\to(C'',d'')\] are filtered chain maps of degrees $j$ and $k$, respectively, then $g\circ f$ is naturally a degree $j+k$ filtered chain map, and \[E_i(g\circ f)=E_i(g)\circ E_i(f)\] for all $i\geq 0$.

\begin{remark}
A degree $k$ filtered chain map $f$ can also be thought of as a degree $j$   map for any $j\leq k$. On the other hand,  the definition of $E_i(f)$ depends on the degree of $f$. It is therefore important  that one specifies the degree of $f$ when talking about these induced maps.
\end{remark}

\begin{remark}
Given a degree $k$ filtered chain map $f$ from $(C,d)$ to $(C',d')$, it is worth pointing out that \[E_{\infty}(f):E_{\infty}(C)\to E_{\infty}(C')\] does not necessarily agree with the the induced map \[f_*:H_*(C,d)\to H_*(C',d'),\] via the isomorphisms between the domains and codomains. In fact, it can be the case that $f_*$ is an isomorphism while $E_{\infty}(f)$ is the zero map e.g. regard the identity map as a degree $-1$ filtered chain map.    What is true, however, is that 
\[ f_*= E_\infty(f)+ \text{higher order terms}\]
where ``higher order terms" means terms  in the decomposition of the adjusted map $f_{(\infty)}=f_*$ according to the grading that shift the grading by more than $k$.

\end{remark}

\begin{remark}
\label{rmk:cancel}
Note that for each cancellation performed in computing the spectral sequence associated to a filtered complex $(C,d)$, the maps $\pi$ and $\iota$ of Lemma \ref{lem:cancel} are degree 0 filtered chain maps. In particular, by taking compositions of these maps, we obtain degree $0$ filtered chain maps \[\pi_{(i)}: (C,d)\to (C_{(i)},d_{(i)})\,\,\,\text{ and }\,\,\,\iota_{(i)}: (C_{(i)},d_{(i)})\to (C,d)\] for each $i\geq 0$.  Tautologically, we  have that the induced maps 
\begin{align*}
E_j(\pi_{(i)})&: E_j(C)\to [E_j(C_{(i)})=E_j(C)]\\
E_j(\iota_{(i)})&: [E_j(C_{(i)})=E_j(C)]\to E_j(C)
\end{align*} are the identity maps for all $j\geq i$.  
\end{remark}

Below, we  prove Lemmas \ref{lem:ssmapiso}, \ref{lem:sseqmaps}, and \ref{lem:einfty} using the above descriptions of spectral sequences and induced maps in terms of cancellation.

\begin{proof}[Proof of Lemma \ref{lem:ssmapiso}] Suppose $f$ is a map as in the lemma and let \[f_{(i)}:(C_{(i)},d_{(i)})\to (C_{(i)}',d'_{(i)})\] be the adjusted map  as defined  above. The fact that $E_i(f)$ is an isomorphism  implies that $f_{(i)}$ is too. Moreover, it is easy to see that its inverse \[g_{(i)}=f_{(i)}^{-1}:(C'_{(i)},d'_{(i)})\to (C_{(i)},d_{(i)})\] is  also a filtered chain map of degree $0$, and that $E_j(f_{(i)})$ and $E_j(g_{(i)})$ are inverses for all $j\geq i$. Let \[g:(C',d')\to(C,d)\]  be the degree $0$ filtered chain map given by $g=\iota_{(i)}\circ g_{(i)}\circ\pi_{(i)}$ for maps \[\pi_{(i)}: (C',d')\to (C'_{(i)},d'_{(i)})\,\,\,\text{ and }\,\,\,\iota_{(i)}: (C_{(i)},d_{(i)})\to (C,d)\] as in Remark \ref{rmk:cancel}. Then   $E_j(f)=E_j(f_{(i)})$ and $E_j(g) = E_j(g_{(i)})$ are inverses for all $j\geq i$. In particular, each $E_j(f)$ is an isomorphism.
\end{proof}

\begin{proof}[Proof of Lemma \ref{lem:sseqmaps}] It is clear from the discussion above  that if a filtered chain map induces the zero map on some page  then it induces the zero map on all subsequent pages. Now suppose $E_i(f) =E_i(g)$ as in the lemma. Then \[E_i(f-g) = E_i(f)-E_i(g)=0,\] which implies that \[E_j(f)-E_j(g)=E_j(f-g) = 0\] for all $j\geq i$, completing the proof.
\end{proof}

\begin{proof}[Proof of Lemma \ref{lem:einfty}]
Note that $(E_i(C),0) = (C_{(i)},d_{(i)})$ in this case. We may therefore take $f$ to be the map \[f=\pi_{(i)}:(C,d)\to(C_{(i)},d_{(i)}), \]  per  Remark \ref{rmk:cancel}.
\end{proof}

\subsection{Khovanov homology}
\label{ssec:kh} In this subsection, we review the definitions and some basic properties of Khovanov homology and its reduced variant. 

Suppose $D$ is a  diagram in $S^2:=\R^2\cup\{\infty\}$ for an oriented link in $S^3:=\R^3\cup\{\infty\}$, with crossings labeled  $1,\dots,n$. Let $n_+$ and $n_-$ denote the numbers of positive and negative crossings of $D$. For  each  $I\in \{0,1\}^n$,  let $I_j$ denote the $j$th coordinate of $I$ and let $D_I$ be the  diagram  obtained by taking the $I_j$-resolution (as shown in  Figure \ref{fig:res}) of the $j$th crossing of $D$, for every $j\in\{1,\dots,n\}$. Let $V(D_I)$ be the vector space generated by the components of $D_I$. We endow $\Lambda^*V(D_I)$ with a grading $\mathbf{p}$ according to the rules  that $1\in \Lambda^0 V(D_I)$ has grading $\mathbf{p}(1)=m$, where $m$ is equal to the number of components of $D_I$, and that wedging with any of the components decreases the $\mathbf{p}$ grading by $2$.
 
 \begin{figure}[!htbp]
 \labellist 
\small\hair 2pt  
\pinlabel $0$ at 101 -6
\pinlabel $1$ at 173 -6 
\endlabellist 
\begin{center}
\includegraphics[width=6.6cm]{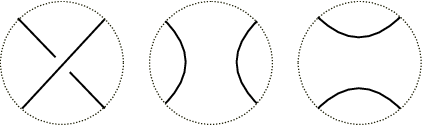}
\caption{\quad The $0$- and $1$-resolutions of a crossing.}
\label{fig:res}
\end{center}
\end{figure}

Given  tuples $I,J\in\{0,1\}^n$, we  write $I<_k J$ if $J$ may be obtained from $I$ by changing exactly $k$ $0$s to $k$ $1$s. For each  pair $I,I'$ with $I<_1 I'$,  one defines a map \[d_{I,I'}:\Lambda^*V(D_I)\to \Lambda^*V(D_{I'}),\] as described below. The Khovanov chain complex assigned to $D$ is then given by \[\cKh(D) = \bigoplus_{I\in\{0,1\}^n}\Lambda^*V(D_I),\] with differential \[d=\bigoplus_{I<_1I'}d_{I,I'}.\] This is a bigraded  complex,  with (co-)\emph{homological grading}  defined  by \[\mathbf{h}(x) = I_1+\dots +I_n-n_-,\]  for   $x\in \Lambda^*V(D_I)$, and  \emph{quantum grading} defined by \[\mathbf{q}(x) = \mathbf{p}(x) + \mathbf{h}(x) + n_+-n_-,\] for homogeneous   $x\in \Lambda^*V(D_I)$.   The differential $d$ increases  $\mathbf{h}$ by one and preserves $\mathbf{q}$. Thus, if we write $\cKh^{i,j}(D)$ for the summand of $\cKh(D)$ in  homological grading $i$ and quantum grading $j$, then $d$ restricts to a differential on \[\cKh^{*,j}(D) = \bigoplus_i\cKh^{i,j}(D)\] for each $j$. We will write \[\Kh^{i,j}(D)=H_i(\cKh^{*,j}(D),d)\] for the (co-)homology of this complex in homological grading $i$. The \emph{Khovanov homology of $D$}  refers to the bigraded vector space \[\Kh(D) = \bigoplus_{i,j}\Kh^{i,j}(D).\]

\begin{remark}
We will also treat the case in which $D$ is the \emph{empty} diagram. In this case, we let $\Kh(D) = \cKh(D)=\Lambda^*(0)= \F$.
\end{remark}

It remains to define  $d_{I,I'}$. Note that   the diagram $D_{I'}$ is obtained from $D_I$ either by merging two circles into one or by splitting one circle into two. Suppose first that $D_{I'}$ is obtained  by merging the components $x$ and $y$ of $D_I$ into one circle. Then there is an obvious  identification \[V(D_{I'})\cong V(D_I)/(x+y),\] and we define the \emph{merge} map $d_{I,I'}$ to be the induced quotient map  \[\Lambda^*V(D_I)\to \Lambda^*(V(D_I)/(x+y))\cong \Lambda^*V(D_{I'}).\] Suppose next that $D_{I'}$ is obtained  by splitting a component of $D_I$ into  two circles $x$ and $y$. Then the  identification \[V(D_I)\cong V(D_{I'})/(x+y)\] induces an identification \[\Lambda^*V(D_I)\cong \Lambda^*(V(D_{I'})/(x+y))\cong (x+y)\wedge \Lambda^*V(D_{I'}),\] and we define the \emph{split} map $d_{I,I'}$ to be the composition of the maps \[\Lambda^*V(D_I)\xrightarrow{\cong} \Lambda^*(V(D_{I'})/(x+y))\xrightarrow{\cong}(x+y)\wedge \Lambda^*V(D_{I'})\xrightarrow{\subset} \Lambda^*V(D_{I'}).\] That is, the split map may be thought of   as given  by wedging with $x+y$.


For diagrams $D$ and $D'$ which differ by a Reidemeister move, Khovanov defines in \cite{kh1} an isomorphism \[\Kh(D)\to\Kh(D'),\] which we  refer to   as the \emph{standard} isomorphism associated to the Reidemeister move. In this way, the isomorphism class of Khovanov homology provides an invariant of oriented link type.

Next, we describe how the theory behaves under disjoint union. Consider the link diagram $D\sqcup D'$ obtained as a disjoint union of diagrams $D$ and $D'$. Suppose $D$ has $m$ crossings and $D'$ has $n$ crossings. For $I\in\{0,1\}^m$ and $I'\in\{0,1\}^n$,  let $II'\in\{0,1\}^{m+n}$ denote the tuple formed via concatenation.  Note that for every such $II'$, there is a canonical  isomorphism \[V((D\sqcup D')_{II'})\to V(D_I)\oplus V(D_{I'}),\] which naturally induces an isomorphism \[\Lambda^*V((D\sqcup D')_{II'})\to \Lambda^*V(D_I)\otimes\Lambda^* V(D_{I'}).\] The direct sum of these  isomorphisms define an isomorphism \[\CKh(D\sqcup D)\to \CKh(D)\otimes \CKh(D'),\] that induces an isomorphism  \[\Kh(D\sqcup D)\to \Kh(D)\otimes \Kh(D'),\] which we  refer to as the \emph{standard} isomorphism associated to disjoint union.

In \emph{reduced} Khovanov homology, one considers \emph{based}  diagrams. These are planar diagrams containing the basepoint $\infty\subset S^2$ (in particular,  all such diagrams are nonempty). Suppose $D$ is such  a based diagram. Consider the chain map \[\Phi_\infty:\cKh(D)\to\cKh(D)\] given on each $V(D_I)$ by wedging with the component of $D_I$ containing  $\infty$. The image of this map is a subcomplex of $\cKh(D)$. The reduced Khovanov complex of $D$ is defined to be  the associated quotient complex, \begin{equation}\label{eqn:reducedkh}\cKhr(D):=(\cKh(D)/{\mathrm{Im}}(\Phi_\infty))[0,-1].\end{equation} The reduced Khovanov homology \[\Khr(D) = H_*(\cKhr(D))\]  is then  the bigraded vector space obtained as  the homology of this quotient complex. In (\ref{eqn:reducedkh}), the bracketed term $[0,-1]$ indicates  a shift of the $(i,j)$ bigrading by $(0,-1)$. This  shift is introduced so that the reduced Khovanov homology of the unknot is supported in bigrading $(0,0)$. 

In reduced Khovanov homology, Reidemeister moves  away from  $\infty$ give rise to isomorphisms of Khovanov groups. In particular,  the isomorphism class of reduced Khovanov homology provides an invariant of \emph{based}, oriented link type. 

Reduced Khovanov homology behaves under disjoint union a little bit differently than Khovanov homology does. In particular, suppose $D$ and $D'$ are disjoint planar diagrams, with $D$ containing $\infty$. Let $U_{\infty}$ denote the small crossingless diagram of the unknot containing $\infty$. Then there is a natural and obvious isomorphism  \[\Khr(D\sqcup D')\to\Khr(D)\otimes\Khr(D' \sqcup U_{\infty}).\] 

\begin{remark}
\label{rmk:khkhr} Note that there is a natural isomorphism between $\Khr(D\sqcup U_{\infty})$ and $\Kh(D)$ for planar diagrams $D$ avoiding $\infty$. \end{remark}

\subsection{Functoriality}
\label{ssec:funct}

In this subsection, we  review some categorical aspects of links, cobordisms, and their diagrams. We then describe how Khovanov homology defines a functor from various cobordism categories to $\vect$.

The  category we will be most interested in is the \emph{link cobordism category} $\link$. Objects of $\link$ are oriented links in $S^3:=\R^3\cup\{\infty\}$ and morphisms are proper isotopy classes of collared, smoothly embedded link cobordisms in $S^3\times[0,1]$. This means that two surfaces represent the same morphism if they differ by smooth isotopy fixing a neighborhood of the boundary pointwise. In order to define a functor from $\link$, one often starts by defining a functor from the \emph{diagrammatic link cobordism category} $\diag$ mentioned in the introduction. This category can be thought of as  a more combinatorial model for $\link$. We define this category below and then describe how functors from $\diag$ can be turned into functors from $\link$, focusing on the case of Khovanov homology.

Objects of $\diag$ are oriented link diagrams in $S^2:=\R^2\cup\{\infty\}$ and morphisms are \emph{movies} up to \emph{equivalence}. We define these two terms below. A \emph{movie} is a 1-parameter family $D_t$, $t\in [0,1]$, where the $D_t$ are link diagrams except at finitely many $t$-values where the topology of the diagram changes by a local move consisting of a Reidemeister move or a Morse modification (a diagrammatic handle attachment). Away from these exceptional $t$-values, the link diagrams  vary by  planar isotopy. Movies $M_1$ and $M_2$ can be composed in a natural way $M_2\circ M_1$, assuming that the initial diagram of $M_2$ agrees with the terminal diagram of $M_1$. Then  any movie can be described as a finite composition of \emph{elementary movies}, where each elementary movie corresponds to either:
\begin{itemize}
\item a Reidemeister move (of type I, II, or III), or
\item an oriented diagrammatic handle attachment (a $0$-, $1$-, or $2$-handle), or
\item a  planar isotopy of diagrams.
\end{itemize}
Carter and Saito \cite{CS} refer to the first two types of elementary movies as {\em elementary string interactions} (ESIs).  We will generally  represent an ESI diagrammatically by recording  diagrams just before and just after the corresponding change in topology. Figure \ref{fig:handleadd} shows the ESIs corresponding to  handle attachments.

\begin{figure}[ht]
\labellist
\hair 2pt
\pinlabel $D$ at -15 105
\pinlabel $D'$ at -15 30

\endlabellist
\centering
\includegraphics[height=4.6cm]{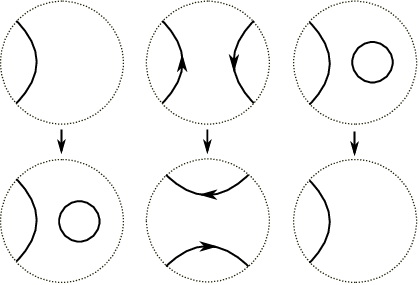}
\caption{ From left to right, oriented diagrammatic $0$-, $1$-, and $2$-handle attachments.}
\label{fig:handleadd}
\end{figure}

Note that a movie $M$  defines an immersed surface $\Sigma_M\subset S^2\times [0,1]$ with \[D_t=\Sigma_M\cap (S^2\times \{t\}).\] We refer to these cross sections as the \emph{levels} of $\Sigma_M$. We will often think of a movie \emph{as} its corresponding immersed surface and vice versa. Let \[\pi:S^3\to S^2\] be the map which sends $\infty$ to $\infty$ and restricts to the projection  \[\pi:\R_{xyz}^3\to\R_{xy}^2\] on the first two coordinates for points in $\R^3\subset S^3$.  Given links $L_0, L_1\subset S^3$ with $\pi(L_i) = D_i$, we can lift  $\Sigma_M$ to a link cobordism $S\subset S^3\times[0,1]$ from $L_0$ to $L_1$ such that \[(\pi\times id) (S) = \Sigma_M.\] As $\diag$ is supposed to serve as a model for $\link$, we ought to declare two movies from $D_0$ to $D_1$ to be equivalent if their  lifts, for fixed $L_0$ and $L_1$,  represent the same morphism in $\link$. Carter and Saito discovered how to interpret this equivalence diagrammatically in \cite{CS}. In particular, two movies  are \emph{equivalent} if they can be related by a finite sequence of the following moves:
\begin{itemize}
\item the \emph{movie moves} of Carter and Saito  \cite[Figures 23-37]{CS},
\item level-preserving isotopies (of their associated immersed surfaces), 
\item interchange of  the levels  containing \emph{distant} ESIs.
\end{itemize}
We will not describe these moves in detail as we do not need them; we refer to the reader to \cite{CS} for more information.

Khovanov homology, as described in the previous subsection, assigns a vector space to a link diagram. To extend Khovanov homology to a functor from $\diag$ to $\vect$, one must assign maps to movies such that equivalent movies are assigned the same map. We describe below how this is done, following Jacobsson  \cite{jacobsson}.

First, one assigns maps to elementary movies. To an elementary movie $M$ from $D_0$ to $D_1$ corresponding to a  Reidemeister move, we assign the associated \emph{standard}  isomorphism \[\Kh(M):\Kh(D_0)\to \Kh(D_1)\] mentioned in the previous subsection. Suppose $M$ is the  movie corresponding to a planar isotopy $\phi$ taking $D_0$ to $D_1$. This isotopy determines a canonical isomorphism  \[F_\phi:\CKh(D_0)\to\CKh(D_1).\] We  assign to $M$ the induced map on homology, \[\Kh(M):=(F_\phi)_*: \Kh(D_0)\to\Kh(D_1).\] It remains to assign a map to a  movie $M$ from $D_0$ to $D_1$ corresponding to an oriented $i$-handle attachment, for $i=0,1,2$. 

For $i=0$, the diagram $D_1$ is a disjoint union  $D_0\sqcup U,$ where $U$ is the crossingless diagram of the unknot. It follows that \[\Kh(D_1)\cong \Kh(D_0)\otimes \Kh(U) = \Kh(D_0)\otimes \Lambda^*(\F\langle U\rangle),\] and we define \[\Kh(M):\Kh(D_0)\to\Kh(D_0)\otimes \Lambda^*(\F\langle U\rangle)\] to be the  map which sends $x$ to $x\otimes 1$ for all $x\in \Kh(D_0)$.

Similarly, for $i=2$, we can view $D_0$ as a disjoint union $D_1\sqcup U$, so that \[\Kh(D_0)\cong\Kh(D_1)\otimes \Lambda^*(\F\langle U\rangle).\] In this case, we define \[\Kh(M): \Kh(D_1)\otimes \Lambda^*(\F\langle U\rangle)\to\Kh(D_0)\] to be the map which sends $x\otimes 1$ to $0$ and $x\otimes U$ to $x$ for all $x\in \Kh(D_1)$.

Finally, for $i=1$, each complete resolution $(D_1)_I$ is obtained from $(D_0)_I$ via a merge or split. The merge and split maps used to define the differential on Khovanov homology therefore give rise to a map \[\Lambda^*V((D_0)_I)\to\Lambda^*V((D_1)_I).\] These maps fit together to define a chain map \[\cKh(D_0)\to\cKh(D_1),\] and  $\Kh(M)$ is the induced map on homology. Put slightly differently, let $\tilde D$ be a diagram with one more crossing than $D_0$ and $D_1$ such that $D_0$ is the $0$-resolution of $\tilde D$ at this  crossing $c$ and $D_1$ is the $1$-resolution (we will think of $c$ as the $(n+1)^{\rm st}$ crossing). Then the Khovanov complex for $\tilde D$ is the mapping cone of the chain map \[T:\cKh(D_1)\to\cKh(D_1),\]  given by the direct sum \[ T = \bigoplus_{I\in \{0,1\}^{n}} d_{I\times\{0\},I\times\{1\}},\] where these \[d_{I\times\{0\},I\times\{1\}}: \Lambda^*V((D_0)_I)\to\Lambda^*V((D_1)_I)\] are components of the differential on $\cKh(\tilde D)$.  Then \[\Kh(M):=T_*:\Kh(D_0)\to \Kh(D_1).\]

Given an arbitrary movie $M$ from $D_0$ to $D_1$, expressed as a composition \[M = M_1\circ\dots\circ M_k\] of elementary movies, we then define \[\Kh(M):\Kh(D_0)\to\Kh(D_1)\] to be the composition \[\Kh(M) = \Kh(M_k)\circ\dots\circ \Kh(M_1).\] In this way, Khovanov homology assigns maps to movies. The key theorem is the following result from \cite{jacobsson}; see also \cite{bncob,khcob}.

\begin{theorem}{\em (Jacobsson \cite{jacobsson})}
If $M$ and $M'$ are equivalent movies, then $\Kh(M) = \Kh(M')$.
\end{theorem}

Jacobsson proves this theorem by showing that the maps assigned to movies are invariant under the moves listed above. As desired, his result implies that Khovanov homology defines a functor \[\Kh:\diag\to\vect.\]

We next consider how to lift this and other functors from $\diag$ to functors from $\link$. We shall achieve this by defining functors, \[ \Pi_\alpha :\link\to \diag.\]
To define $\Pi_\alpha$, we take for every link $L\subset S^3$ a choice of smooth isotopy $\phi^{\alpha}_L$ which begins at $L$ and ends at a link $\phi^\alpha_L(L)$ on which the projection map  \[\pi:S^3\to S^2\] restricts to a regular immersion. We will also regard such an isotopy as a morphism \[\phi^\alpha_L\in\mor(L,\phi^\alpha_L(L)),\] represented by  the smoothly embedded cylinder obtained from its trace.   On objects, we define $\Pi_\alpha$ by
\[ \Pi_\alpha(L):=\pi(\phi^\alpha_L(L)).\]
Given a morphism $S\in \mor(L_0,L_1)$, let us consider the associated morphism \[\phi^\alpha_{L_1}\circ S \circ (\phi^\alpha_{L_0})^{-1} \in \mor (\phi^\alpha_{L_0}(L_0),\phi^\alpha_{L_1}(L_1)).\] According to \cite[Theorem 5.2, Remark 5.2.1(2)]{CS}, there is a representative $\Sigma$ of this morphism whose image under the projection \[\pi\times id:S^3\times [0,1]\to S^2\times[0,1]\] is a movie. We define $\Pi_\alpha(S)$ to be the equivalence class of this movie, \[\Pi_\alpha(S):=[(\pi\times id)(\Sigma)].\] 
\begin{prop} $\Pi_\alpha:\link\to \diag$ is a functor.
\end{prop}
\begin{proof} Clearly $\Pi_\alpha$ is well-defined on objects.  To see that it is well-defined on morphisms, we use the relative version of Carter and Saito's main result \cite[Theorem 7.1]{CS}, which states that properly isotopic surfaces project to equivalent  movies.  Thus,  the movies resulting from the projections of any two representatives of the morphism $\phi^\alpha_{L_1}\circ S\circ (\phi^\alpha_{L_0})^{-1}$  are  equivalent.
 \end{proof}
The apparent dependence of the functor $\Pi_\alpha$ on the choices of isotopies $\phi_L^\alpha$ is undesirable.  In fact, we have the following.
\begin{prop}\label{prop:natiso}Suppose that $\{\phi_L^\alpha \}$ and $\{\phi_L^\beta \}$ are two collections of isotopies to links with regular projections, as above, defining functors \[\Pi_\alpha, \Pi_\beta:\link\to\diag.\]   Then the functors $\Pi_\alpha$ and $\Pi_\beta$ are naturally isomorphic.
\end{prop}
\begin{proof}  The assignment $\theta_\alpha^\beta$ which sends a link $L$ to the morphism \[\theta_\alpha^\beta(L):=[(\pi\times id)(\Sigma)]\in\mor(\Pi_\alpha(L),\Pi_\beta(L)),\] where $\Sigma$ is a representative of the morphism $\phi^\beta_L\circ (\phi^\alpha_L)^{-1}$ whose image under $\pi\times id$ is a movie, gives a well-defined  natural isomorphism from $\Pi_\alpha$ to $\Pi_\beta$.  Commutativity of the square
 \[ \xymatrix@C=30pt@R=25pt{
 \Pi_\alpha(L_0) \ar[r]^-{\Pi_\alpha(S)} \ar[d]_{\theta_\alpha^\beta(L_0)} & \Pi_\alpha(L_1) \ar[d]^{\theta_\alpha^\beta(L_1)} \\
 \Pi_\beta(L_0) \ar[r]_-{\Pi_\beta(S)} &\Pi_\beta(L_1)} \]
 follows from the work of Carter and Saito; we leave it as an exercise. It is also not hard to show that $\theta_\beta^\alpha$ is the inverse natural transformation, and that \[\theta_\beta^\gamma \circ \theta_\alpha^\beta = \theta_\alpha^\gamma\] for any three collections of isotopies.
\end{proof}Moreover we have
\begin{prop} For any choice of isotopies $\phi_L^\alpha$, the functor $\Pi_\alpha:\link\to \diag$ is an equivalence of categories.
\end{prop}
\begin{proof}  Since any two such functors are naturally isomorphic, it is enough to verify the proposition for a good choice of isotopies $\phi_L^\alpha$.  We take isotopies $\phi_L^\alpha$ such that if $L$ is already regularly immersed under the map $\pi$, then $\phi_L^\alpha$ is the identity isotopy.
Hence we have that $\Pi_\alpha$ is surjective on objects.  Furthermore, $\Pi_\alpha$ is bijective on morphism sets (that is, it  is full and faithful), which suffices to establish the equivalence by, e.g. \cite[Theorem 1, IV.4]{MacLane}.  Surjectivity on morphisms is easy since movies can easily be lifted to cobordisms in $S^3\times [0,1]$, whereas injectivity on morphisms is again a consequence of \cite[Theorem 7.1]{CS}.
\end{proof}

One can then lift Khovanov homology to a functor from $\link$  by precomposing with any $\Pi_\alpha$.  We shall denote this functor by  \[\Kh_\alpha:=\Kh\circ\Pi_\alpha:\link\to\vect.\] This functor assigns vector spaces to links, but these vector spaces  depend on extra data, the extra data being the set of isotopies $\{\phi^\alpha_L\}_{L\subset S^3}$ to links with regular projections. We would prefer a functor which assigns vector spaces to links  themselves, and does not depend on the choice of isotopies.  Our solution rests on the natural isomorphisms we have described between the functors $\Pi_\alpha$.

Indeed, using notation from the proof of Proposition \ref{prop:natiso}, we obtain  isomorphisms
\[ Kh_\alpha^\beta(L):=Kh(\theta_\alpha^\beta(L)):  Kh_\alpha(L)\to Kh_\beta(L)\] satisfying
$Kh_\beta^\gamma\circ Kh_\alpha^\beta= Kh_\alpha^\gamma$  and $Kh_\alpha^\alpha=\text{Id}$, for all $\alpha,\beta,\gamma$.  Thus the collection of vector spaces $\{Kh_\alpha(L)\}_\alpha$ and isomorphisms $\{Kh_{\alpha}^\beta(L)\}_{\alpha,\beta}$ form a transitive system in the sense of \cite[Chapter 1.6]{eilenbergsteenrod}.  We  define $Kh(L)$ to be the vector space associated to this system in the standard way, as the inverse limit over the complete directed graph on  the set of isotopies.  A morphism $S\in\mor(L_0,L_1)$ then gives rise to a well-defined map \[\Kh(S):\Kh(L_0)\to\Kh(L_1),\] so that $\Kh$ defines a functor from $\link$ to $\vect$ which is independent of any choice of isotopies, as desired.
 \begin{remark}
	\label{rmk:linkvsisotopyclass}  Of course, one should not expect that $Kh$ can be lifted to a functor associating a vector space to each \emph{isotopy class} of link.  This is because there exist knots with self-isotopies inducing non-identity automorphisms of the Khovanov invariant.
\end{remark}
We conclude this section by noting that reduced Khovanov homology defines a similar functor  \[\Khr:\link_\infty\to\vect\] from the \emph{based link cobordism category} $\link_\infty$. Objects of $\link_\infty$ are oriented links in $S^3$ containing the basepoint $\infty$ and morphisms are proper isotopy classes of collared, smoothly embedded link cobordisms in $S^3\times[0,1]$ containing the arc $\{\infty\}\times [0,1]$. In particular, two surfaces represent the same morphism if they differ by smooth isotopy fixing a neighborhood of the boundary and this arc pointwise. In order to define the functor $\Khr$ above, one first defines a functor from the \emph{based diagrammatic link cobordism category} $\diag_\infty$. Objects of this category are equivalence classes of \emph{based} movies in which each $D_t$ contains $\infty$. Any such movie can be expressed as a composition of elementary movies corresponding to Reidemeister moves, handle attachments, and planar isotopies, all supported away from $\infty$. Two based movies are considered equivalent if they are related by obvious based versions of moves from before. To define a functor \[\Khr:\diag_{\infty} \to\vect\] one then associates maps to elementary based movies and proceeds as before, noting that Jacobsson's work implies that equivalent based movies are assigned the same map. One then promotes this to a functor from $\link_\infty$ by a straightforward adaptation of the ideas above.

\begin{remark} 
\label{rmk:promotion} It is clear that a similar procedure works for promoting any functor from $\diag$ to $\spect$ to a  functor from $\link$ to $\spect$, and similarly for the based categories. With this in mind, we will be content to  work  solely in the diagrammatic categories in the rest of this paper.
\end{remark}

\section{Khovanov-Floer theories}
\label{sec:defs}
In this section, we give a precise definition of a  \emph{Khovanov-Floer theory} (and its reduced variant) and describe what it means for such a theory to be \emph{functorial}. The main challenge lies in setting up the right algebraic framework, as is illustrated by  thinking about Kronheimer and Mrowka's spectral sequence in singular instanton knot homology. The difficulty is that their construction does not associate a filtered chain complex to a link diagram alone, but to a link diagram  together with some auxiliary data (e.g. families of metrics and perturbations), so it is  not immediately obvious in what sense  the resulting spectral sequence gives an assignment of objects in $\spect$ to link diagrams.  The same is true in Ozsv{\'a}th and Szab{\'o}'s work. Indeed, Kronheimer and Mrowka's construction assigns to a  diagram $D$ and a choice of data $\adata$ a filtered chain complex \[C^\adata(D)=(C^\adata(D),d^\adata(D))\footnote{We will often leave out the differential in the the notation for a chain complex.}\] and an isomorphism of vector spaces
\[q^\adata: \Kh(D)\to E_2(C^\adata(D)).\] Any two choices of auxiliary data $\adata,\adata'$ result in what one might call \emph{quasi-isomorphic} constructions, in that there exists a filtered chain map \[f:C^\adata(D)\to C^{\adata'}(D)\] such that \[E_2(f)=q^{\adata'}\circ (q^\adata)^{-1}\] (which implies that $f$ is a quasi-isomorphism by the results of Subsection \ref{ssec:homalg}). So, really, one would  like to say that what Kronheimer and Mrowka's construction assigns to a link diagram $D$ is a \emph{quasi-isomorphism class} of pairs $(C^\adata(D),q^\adata)$. The algebraic framework introduced below is meant to make this idea  meaningful.

Given a graded vector space $V$, we define a \emph{$V$-complex}  to be a pair $(C,q)$, where $C$ is a filtered chain complex and \[q:V\to E_2(C)\] is a grading-preserving isomorphism of vector spaces. Suppose $(C,q)$ and $(C',q')$ are $V$- and $W$-complexes, and  let \[T:V\to W\] be a homogeneous degree $k$  map of graded vector spaces. A \emph{morphism from $(C,q)$ to $(C',q')$ which agrees on $E_2$ with $T$} is a degree $k$ filtered chain map \[f:C\to C'\] such that \[E_2(f) =q'\circ T\circ q^{-1}.\] Note that  if $f$ and $g$ are two such morphisms, then $E_i(f) = E_i(g)$ for $i=2$ and, therefore, for all $i\geq 2$ by Lemma \ref{lem:sseqmaps}. 
A \emph{quasi-isomorphism} is a morphism from one $V$-complex  to another  which agrees on $E_2$ with the identity map on $V$. 

\begin{remark}
Note that the existence of a quasi-isomorphism from $(C,q)$ to $(C',q')$ implies the existence of a quasi-isomorphism from $(C',q')$ to $(C,q)$ by Lemma \ref{lem:ssmapiso}. 
\end{remark}

For any two quasi-isomorphisms \[f,g:(C,q)\to(C',q'),\] we have that $E_i(f) = E_i(g)$ for all $i\geq 2$ by the discussion above. Moreover, given quasi-isomorphisms \[f:(C,q)\to(C',q')\, \,\,\text{ and }\,\,\,g:(C',q')\to(C'',q''),\] we have that  \[E_i(g\circ f) = E_i(g)\circ E_i(f)\] for all $i\geq 2$. In other words, the higher pages in the spectral sequences associated to quasi-isomorphic $V$-complexes are  \emph{canonically} isomorphic as vector spaces, and, since the $E_i(f)$ are chain maps, these higher pages are also canonically isomorphic \emph{as chain complexes}. More precisely, for each $i\geq 2$, the collection of chain complexes $(E_i,d_i)$ associated with representatives of a given quasi-isomorphism class $\cA$ of $V$-complexes fits into a transitive system of chain complexes, from which one can extract an honest chain complex by taking the inverse limit. In summary, then, a quasi-isomorphism class $\cA$ of $V$-complexes therefore determines  a well-defined  graded chain complex $(E_i(\cA),d_i(\cA))$ for each $i\geq 2$. This is the sense in which, for example, Kronheimer and Mrowka's construction provides an assignment of objects in $\spect$ to link diagrams.


Now suppose $\cA$ is a quasi-isomorphism class of $V$-complexes and $\cB$ is a quasi-isomorphism class of $W$-complexes, and let \[T:V\to W\] be a homogeneous degree $k$ map of vector spaces. We will say that \emph{there exists a morphism from $\cA$ to $\cB$ which agrees on $E_2$ with $T$} if there exists a morphism \begin{equation}\label{eqn:ssmorf} f:(C,q)\to (C',q')\end{equation} which agrees on $E_2$ with $T$ for some representatives $(C,q)$ of $\cA$ and $(C',q')$ of $\cB$. The morphism in (\ref{eqn:ssmorf}) gives rise to a  homogeneous degree $k$ map \begin{equation}\label{eqn:ssmor}E_i(\cA)\to E_i(\cB)\end{equation} for each $i\geq 2$. Furthermore, this map is  independent of the representative morphism in (\ref{eqn:ssmorf}) in the sense that   if $(C'',q'')$ and $(C''',q''')$ are other representatives of $\cA$ and $\cB$ and \[f':(C'',q'')\to (C''',q''')\] is another morphism which agrees on $E_2$ with $T$, then the diagram \[ \xymatrix@C=30pt@R=25pt{
E_i(C) \ar[r]^-{E_i(f)} \ar[d]_{} & E_i(C') \ar[d]^{} \\
E_i(C'') \ar[r]_-{E_i(f')} & E_i(C''')
} \] commutes for each $i\geq 2,$ where the vertical arrows indicate the canonical isomorphisms between these higher  pages. In summary,  the existence of a morphism from $\cA$ to  $\cB$ which agrees on $E_2$ with $T$ \emph{canonically} determines a   chain map from $(E_i(\cA),d_i(\cA))$ to $(E_i(\cB),d_i(\cB))$ for all $i\geq 2$.

The discussion above shows that quasi-isomorphism classes of $V$-complexes behave  exactly like honest filtered chain complexes with regard to the spectral sequences they induce. This will enable us to bypass the sort of technical difficulty mentioned at the beginning of this subsection for the spectral sequences defined by Kronheimer-Mrowka and Ozsv{\'a}th-Szab{\'o}.

Finally, note that if $(C,q)$ is a  $V$-complex and $(C',q')$ is a $W$-complex, then there is a natural tensor product in the form of a $(V\otimes W)$-complex $(C\otimes C',q\otimes q')$. This  extends in the obvious way to a notion of tensor product between quasi-isomorphism classes of $V$- and $W$-complexes.





Below, we define the term  \emph{Khovanov-Floer theory}. In the definition, we are thinking of the vector space $\Kh(D)$  as being singly-graded by  some  linear combination of the homological and quantum gradings. We will omit this linear combination from the notation. In practice, it will depend on the theory of interest:  we will use the homological grading for the spectral sequence constructions of Kronheimer-Mrowka, Ozv{\'a}th-Szab{\'o}, and Szab{\'o}, and the quantum grading for Lee's construction.

\begin{definition}
\label{def:kftheory}
A \emph{Khovanov-Floer theory} $\cA$ is a rule which assigns to every  link diagram $D$ a quasi-isomorphism class  of $\Kh(D)$-complexes $\cA(D)$, such that: 
\begin{enumerate}
\item if $D'$ is obtained from $D$ by a  planar isotopy, then there exists a morphism \[\cA(D)\to \cA(D')\] which agrees on $E_2$ with the  induced map from $\Kh(D)$ to $\Kh(D')$;
\item if $D'$ is obtained from $D$ by a  diagrammatic $1$-handle attachments, then there exists a morphism \[\cA(D)\to \cA(D')\] which agrees on $E_2$ with the  induced map from $\Kh(D)$ to $\Kh(D')$;
\item  for any two diagrams $D,D'$, there exists a morphism \[\cA(D\sqcup D')\to \cA(D)\otimes \cA(D')\] which agrees on $E_2$ with the standard isomorphism  \[\Kh(D\sqcup D')\to\Kh(D)\otimes \Kh(D');\]
\item for any diagram $D$ of the unlink, $E_2(\cA(D)) = E_{\infty}(\cA(D))$.
\end{enumerate}
\end{definition}

A  \emph{reduced Khovanov-Floer theory} is defined almost exactly as above, except that all link diagrams are now based; planar isotopies and $1$-handle attachments fix and avoid the basepoint $\infty$, respectively;  we  replace all occurrences of $\Kh$ with $\Khr$;  and we replace condition $(3)$ with the condition that for disjoint diagrams $D$ and $D'$, where $D$ contains $\infty$,  there exists a morphism \[\cA(D\sqcup D')\to \cA(D)\otimes \cA(D'\sqcup U_\infty)\] which agrees on $E_2$ with the standard isomorphism  \[\Khr(D\sqcup D')\to\Khr(D)\otimes \Khr(D'\sqcup U_{\infty})\] described at the end of Subsection \ref{ssec:kh}.




\begin{remark}
Note that if $\cA$ is a reduced Khovanov-Floer theory, then the assignment \[D\mapsto\cA(D\sqcup U_\infty)\] (we can assume $D$ avoids $\infty$ after small perturbation) naturally defines a Khovanov-Floer theory,  via the relationship between reduced and unreduced Khovanov homology mentioned in Remark \ref{rmk:khkhr}.
\end{remark}


An immediate consequence of these definitions is that a Khovanov-Floer theory $\cA$  assigns a \emph{canonical} morphism of spectral sequences \[\{(E_i(\cA(D)),d_i(\cA(D))\to (E_i(\cA(D'),E_i(\cA(D'))\}_{i\geq 2}\]  to a movie corresponding to a planar isotopy or  diagrammatic $1$-handle attachment. Of course, we wish to show, in proving Theorem \ref{thm:khfunct}, that  a Khovanov-Floer theory assigns a morphism of spectral sequences to \emph{any} movie, such that equivalent movies are assigned the same morphism. This leads to the definition below.

\begin{definition}
\label{def:kfunct}
A Khovanov-Floer theory $\cA$ is  \emph{functorial} if, given a movie from $D$ to $D'$, there exists a morphism \[\cA(D)\to \cA(D')\] which agrees on $E_2$ with the induced map from $\Kh(D)$ to $\Kh(D')$.
\end{definition}

Thus, a functorial Khovanov-Floer theory assigns a canonical morphism of spectral sequences \[\{(E_i(\cA(D)),d_i(\cA(D))\to (E_i(\cA(D'),E_i(\cA(D'))\}_{i\geq 2}\]  to any movie, which agrees on $E_2$ with the corresponding movie map on Khovanov homology. It follows that equivalent movies are assigned the same morphism since they are assigned the same map in Khovanov homology. In other words, the spectral sequence associated with a functorial Khovanov-Floer theory defines a functor from $\diag$ to $\spect$ and, therefore, by Subsection \ref{ssec:funct}, a functor \[F:\diag\to\spect\] satisfying $SV_2\circ F = \Kh$. (In particular,  the higher pages of the spectral sequence associated with a functorial Khovanov-Floer theory are link type invariants.)
Thus,  in order to prove Theorem \ref{thm:khfunct}, it suffices to prove the following theorem, which we do in the next section.

\begin{theorem}
\label{thm:kfunct2}
Every  Khovanov-Floer theory is functorial. 
\end{theorem}

\begin{remark}
\label{rmk:weakstrong}
The rather simple conditions in the definition of a Khovanov-Floer theory may be thought of as a sort of \emph{weak functoriality}. In practice, it is often relatively easy to verify that a theory satisfies these conditions (we will provide several such verifications in Section \ref{sec:kfthys}). In contrast, functoriality has not been verified for any of spectral sequence constructions that we know of.  Reidemeister invariance has been established in a number of cases (including for the spectral sequences we consider in this paper), but the arguments are generally adapted to the particular theory under consideration. Our approach is more universal. In particular, Theorem \ref{thm:kfunct2} may  be interpreted  as saying that weak functoriality implies functoriality.
\end{remark}

There is an obvious analogue of Definition \ref{def:kfunct} for reduced Khovanov-Floer theories, involving based movies and reduced Khovanov homology. The corresponding analogue of Theorem \ref{thm:kfunct2}, that every reduced Khovanov-Floer theory is functorial, also holds by essentially the same proof.

\section{Khovanov-Floer theories are functorial}
\label{sec:proofs}
This section is dedicated to proving Theorem \ref{thm:kfunct2} (and, therefore, Theorem \ref{thm:khfunct}).

Suppose $\cA$ is a  Khovanov-Floer theory. We will prove below that $\cA$ is functorial. We first show that $\cA$ assigns a canonical morphism of spectral sequences to the movie corresponding to \emph{any} diagrammatic handle attachment (as opposed to only $1$-handle attachments). This follows immediately from the proposition below.

\begin{proposition}
\label{prop:02handle} If $D'$ is obtained from $D$ by a  diagrammatic handle attachment, then there exists a morphism \[\cA(D)\to \cA(D')\] which agrees on $E_2$ with the induced map from $\Kh(D)$ to $\Kh(D')$.
\end{proposition}

\begin{proof}[Proof of Proposition \ref{prop:02handle}]
The $1$-handle case is  part of the definition of a Khovanov-Floer theory. Suppose $D'$ is obtained from $D$ by a $0$-handle attachment. Then $D'=D\sqcup U$. Thus, by condition (3) in Definition \ref{def:kftheory}, there exists a morphism 
\begin{equation}\label{eqn:0handle}\cA(D)\otimes\cA(U)\to \cA(D')\end{equation}
which agrees on $E_2$ with the standard isomorphism  
\[g_2:\Kh(D)\otimes\Kh(U)\to \Kh(D').\] Condition (4) in Definition \ref{def:kftheory} says that 
\[
E_\infty(\cA(U))=E_2(\cA(U))\cong\Kh(U).
\]
It is then an easy consequence of Lemma \ref{lem:einfty} that  the quasi-isomorphism class $\cA(U)$ contains the \emph{trivial} $\Kh(U)$-complex $(\Kh(U),id)$. It follows   that there exists a morphism 
\begin{equation}\label{eqn:0handle2}\cA(D)\to\cA(D)\otimes\cA(U)\end{equation}
which agrees on $E_2$ with the  isomorphism 
\[g_1:\Kh(D)\to\Kh(D)\otimes \Kh(U)\] which sends $x$ to $x\otimes 1$. Indeed, if $(C,q)$ is a representative of $\cA(D)$, then $(C\otimes \Kh(U),q\otimes id)$ is a representative of $\cA(D)\otimes\cA(U)$ and the morphism \[(C,q)\to (C\otimes\Kh(D),q\otimes id)\] which sends $x$ to $x\otimes 1$ is a representative of the desired morphism in (\ref{eqn:0handle2}).
Let $f_1$ and $f_2$ be representatives of the morphisms in (\ref{eqn:0handle2}) and (\ref{eqn:0handle}), respectively. Then the composition $f_2\circ f_1$ from a representative of $\cA(D)$ to a representative of $\cA(D')$ is a morphism which agrees on $E_2$ with the composition \[g_2\circ g_1:\Kh(D)\to \Kh(D'),\] and this latter composition is precisely the map on Khovanov homology associated to the $0$-handle attachment. 
The $2$-handle case is virtually identical.
\end{proof}

Next, we  show that $\cA$ assigns a canonical morphism of spectral sequences to the movie corresponding to a Reidemeister move. This follows immediately from the proposition below.

\begin{proposition}
\label{prop:reid} If $D'$ is obtained from $D$ by a Reidemeister move, then there exists a morphism \[\cA(D)\to \cA(D')\] which agrees on $E_2$ with the standard isomorphism  from $\Kh(D)$ to $\Kh(D')$.
\end{proposition}

Before proving Proposition \ref{prop:reid}, let us first assume this proposition is true and prove    Theorem \ref{thm:kfunct2}.

\begin{proof}[Proof of Theorem \ref{thm:kfunct2}] Suppose $M$ is a movie from $D$ to $D'$. Express this movie as a composition  \[M = M_k\circ \dots \circ M_1,\] where each $M_i$ is an elementary movie from a diagram $D_i$ to a diagram $D_{i+1}$. Let $f_i$ be a morphism from a representative of $\cA(D_i)$ to a representative of $\cA(D_{i+1})$ which agrees on $E_2$ with the corresponding map on Khovanov homology. For the elementary movies corresponding to planar isotopies, such maps exist by Definition \ref{def:kftheory}. For those corresponding to diagrammatic handle attachments or Reidemeister moves, such maps exist by Propositions \ref{prop:02handle} and \ref{prop:reid}. The composition \[f_{k}\circ\dots \circ f_1\] is therefore a morphism from a representative of $\cA(D=D_1)$ to a representative of $\cA(D'=D_{k+1})$ which agrees on $E_2$ with  the map on Khovanov homology induced by this movie.  This proves that $\cA$ is functorial.
\end{proof}

It therefore only remains to prove Proposition \ref{prop:reid}. We break this verification into  three lemmas---one for each type of Reidemeister move. The idea common to the proofs of all  three lemmas is, as mentioned in the introduction,  to  arrange via movie moves that  the Reidemeister move  takes place amongst unknotted components.  This idea was used by the third author in \cite{lobb} in showing that a generic perturbation of Khovanov-Rozansky homology gives rise to a lower-bound on the slice genus. 



\begin{lemma}
\label{lem:R1}
Suppose $D'$ is obtained from $D$ by a Reidemeister I move. Then there exists a morphism \[\cA(D)\to\cA(D')\] which agrees on $E_2$ with the standard isomorphism from $\Kh(D)$ to $\Kh(D')$. \end{lemma}

\begin{proof} Consider the  link diagrams  shown in Figure \ref{fig:R1}. The arrows in this figure are meant to indicate the fact that the movie represented by  the sequence of diagrams \[D=D_1,D_2,D_3,D_4=D',\] as indicated by the thin arrows, is equivalent to the movie consisting of the single Reidemeister I move from $D$ to $D'$, as indicated by the thick arrow. These two movies therefore induce the same map from $\Kh(D)$ to $\Kh(D')$. Thus, to prove Lemma \ref{lem:R1}, it suffices to prove that there exist morphisms 
\begin{align}
\label{eqn:mor1}\cA(D_1)\to\cA(D_2)\\
\label{eqn:mor2}\cA(D_2)\to\cA(D_3)\\
\label{eqn:mor3}\cA(D_3)\to\cA(D_4)
\end{align}
which agree on $E_2$ with the corresponding maps on Khovanov homology. The top and bottom arrows in Figure \ref{fig:R1} correspond to $0$- and $1$-handle attachments; therefore, the morphisms  in (\ref{eqn:mor1}) and (\ref{eqn:mor3}) exist  by Proposition \ref{prop:02handle}. It remains to show that the morphism in (\ref{eqn:mor2}) exists.

\begin{figure}[ht]
\labellist
\hair 2pt
\pinlabel $D=D_1$ at 135 535
\pinlabel $D_2$ at 406 535
\pinlabel $D_3$ at 405 10
\pinlabel $D'=D_4$ at 132 10

\endlabellist
\centering
\includegraphics[height=7cm]{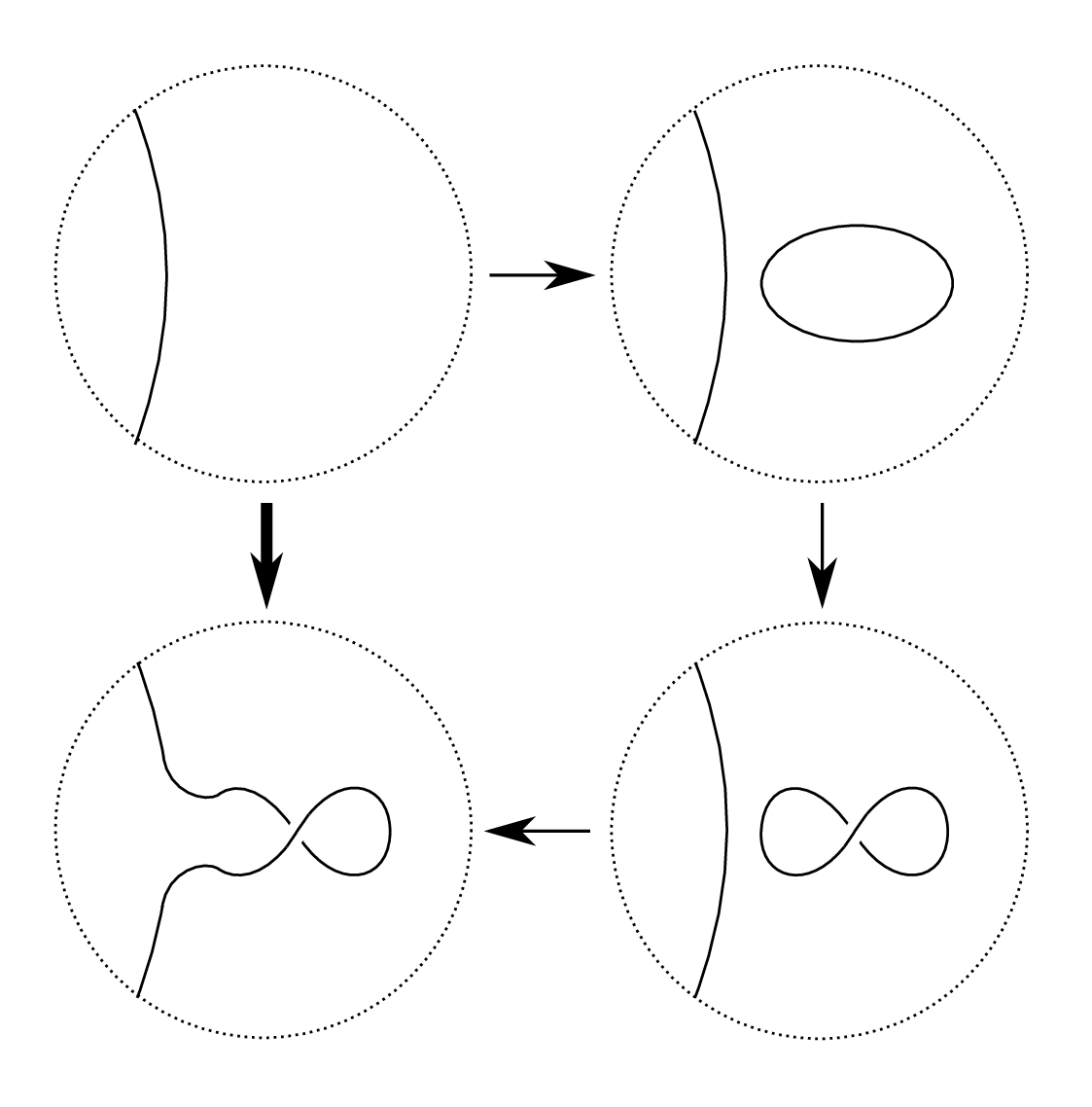}
\caption{ The diagrams $D=D_1,\dots,D_4=D'$. The movie indicated by the thin arrows is equivalent to the movie corresponding to the Reidemeister I move,  indicated by the thick arrow.}
\label{fig:R1}
\end{figure}
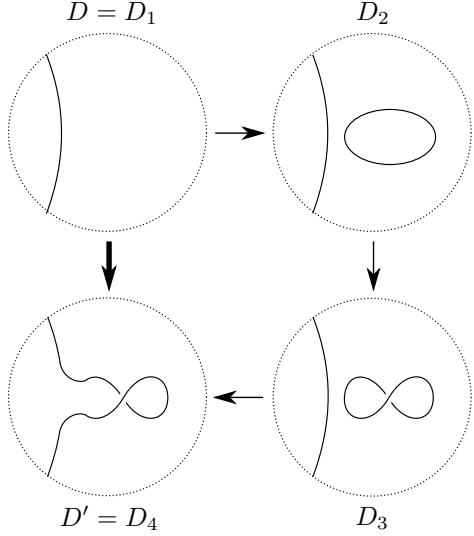

Let $U_0$ and $U_1$ be the  $0$- and $1$-crossing diagrams of the unknot in  $D_2$ and $D_3$, so that $D_2=D_1\sqcup U_0$ and $D_3 = D_1\sqcup U_1$. Thus, by condition (3) in Definition \ref{def:kftheory}, there exist morphisms 
\begin{align}
\label{eqn:du1}&\cA(D_2)\to\cA(D_1)\otimes \cA(U_0)\\
\label{eqn:du2}&\cA(D_1)\otimes \cA(U_1)\to \cA(D_3)
\end{align}
which agree on $E_2$ with the standard isomorphisms  
\begin{align*}
g_1&:\Kh(D_2)\to\Kh(D_1)\otimes \Kh(U_0)\\
g_3&:\Kh(D_1)\otimes \Kh(U_1)\to \Kh(D_3).
\end{align*}
Condition (4) in Definition \ref{def:kftheory} says that 
\begin{align*}
E_\infty(\cA(U_i))&=E_2(\cA(U_i))\cong\Kh(U_i)
\end{align*}
for $i=0,1$, which implies, just as in the proof of Proposition \ref{prop:02handle}, that  the quasi-isomorphism class $\cA(U_i)$ contains the trivial $\Kh(U_i)$-complex $(\Kh(U_i),id)$ for $i=0,1$. It follows immediately that there exists a morphism 
\begin{equation}\label{eqn:u0u1}\cA(U_0)\to\cA(U_1)\end{equation} 
which agrees on $E_2$ with the standard isomorphism 
\begin{equation*}\label{eqn:ku0u1}g_2:\Kh(U_0)\to\Kh(U_1)\end{equation*}
associated to the Reidemeister I move relating these two diagrams of the unknot. Let $f_1,$ $f_2,$ and $f_3$ be representatives of the morphisms in (\ref{eqn:du1}), (\ref{eqn:u0u1}), and (\ref{eqn:du2}), respectively. Then the composition \[f_3\circ (id\otimes f_2)\circ f_1\] from a representative of $\cA(D_2)$ to a representative of $\cA(D_3)$ is a morphism which agrees on $E_2$ with the composition \begin{equation*}\label{eqn:maponkh}g_3\circ (id\otimes g_2)\circ g_1:\Kh(D_2)\to\Kh(D_3),\end{equation*} and this latter composition is equal to the isomorphism from $\Kh(D_2)$ to $\Kh(D_3)$ associated to the Reidemeister I move. It follows  that the morphism in (\ref{eqn:mor2}) exists.
\end{proof}

\begin{lemma}
\label{lem:R2}
Suppose $D'$ is obtained from $D$ by a Reidemeister II move. Then there exists a morphism \[\cA(D)\to\cA(D')\] which agrees on $E_2$ with the standard isomorphism from $\Kh(D)$ to $\Kh(D')$. \end{lemma}

\begin{proof}
Consider the  link diagrams  shown in Figure \ref{fig:R2}. The  arrow from $D=D_1$ to $D_2$ represents two $0$-handle attachments; the  arrow from $D_2$ to $D_3$ represents a Reidemeister II move; and the  arrow from $D_3$ to $D_4=D'$ represents two $1$-handle attachments. The movie represented by these thin arrows  is equivalent to the movie from $D$ to $D'$ corresponding to the single Reidemeister II move indicated by the thick arrow. These two movies therefore induce the same map from $\Kh(D)$ to $\Kh(D')$. Thus, to prove Lemma \ref{lem:R2}, it suffices to prove that there exist morphisms 
\begin{align}
\label{eqn:mor12}\cA(D_1)\to\cA(D_2)\\
\label{eqn:mor22}\cA(D_2)\to\cA(D_3)\\
\label{eqn:mor32}\cA(D_3)\to\cA(D_4)
\end{align}
which agree on $E_2$ with the corresponding maps on Khovanov homology. 
Since the top and bottom arrows in Figure \ref{fig:R2} correspond to handle attachments, the morphisms  in (\ref{eqn:mor12}) and (\ref{eqn:mor32}) exist  by Proposition \ref{prop:02handle}. It remains to show that the morphism in (\ref{eqn:mor22}) exists.

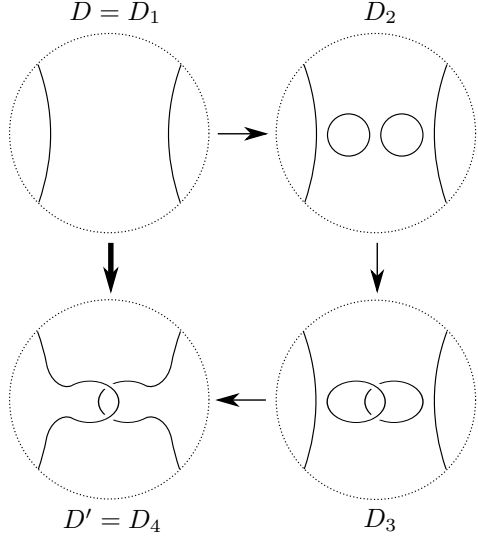
\begin{figure}[ht]
\labellist
\hair 2pt
\pinlabel $D=D_1$ at 122 532
\pinlabel $D_2$ at 395 532
\pinlabel $D_3$ at 395 9
\pinlabel $D'=D_4$ at 119 9

\endlabellist
\centering
\includegraphics[height=7cm]{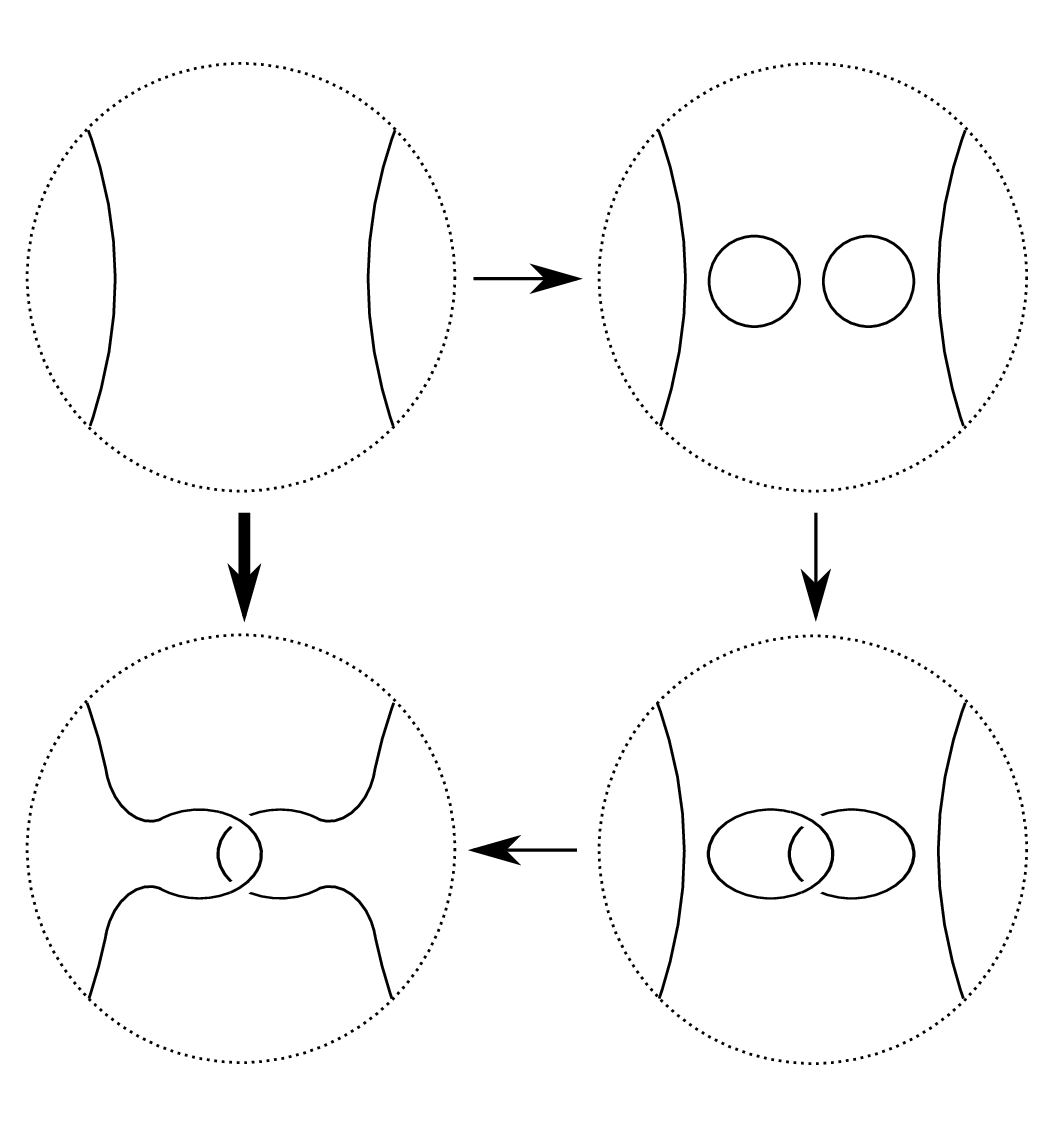}
\caption{ The diagrams $D=D_1,\dots,D_4=D'$. The movie indicated by the thin arrows is equivalent to the movie corresponding to the Reidemeister II move,  indicated by the thick arrow.}
\label{fig:R2}
\end{figure}

Let $U_0$ and $U_2$ be the  $0$- and $2$-crossing diagrams of the $2$-component unlink  in  $D_2$ and $D_3$, so that $D_2=D_1\sqcup U_0$ and $D_3 = D_1\sqcup U_2$. By condition (3) in Definition \ref{def:kftheory}, there exist morphisms 
\begin{align}
\label{eqn:du12}&\cA(D_2)\to\cA(D_1)\otimes \cA(U_0)\\
\label{eqn:du22}&\cA(D_1)\otimes \cA(U_2)\to \cA(D_3)
\end{align}
which agree on $E_2$ with the standard isomorphisms  
\begin{align*}
g_1&:\Kh(D_2)\to\Kh(D_1)\otimes \Kh(U_0)\\
g_3&:\Kh(D_1)\otimes \Kh(U_1)\to \Kh(D_3).
\end{align*}
Condition (4) in Definition \ref{def:kftheory} says that 
\begin{align*}
E_\infty(\cA(U_i))&=E_2(\cA(U_i))\cong\Kh(U_i)
\end{align*}
for $i=0,2$, which implies as in the previous proof that  the quasi-isomorphism class $\cA(U_i)$ contains the trivial $\Kh(U_i)$-complex $(\Kh(U_i),id)$ for $i=0,2$. It follows immediately that there exists a morphism 
\begin{equation}\label{eqn:u0u12}\cA(U_0)\to\cA(U_2)\end{equation} 
which agrees on $E_2$ with the standard isomorphism 
\begin{equation*}\label{eqn:ku0u12}g_2:\Kh(U_0)\to\Kh(U_2)\end{equation*}
associated to the Reidemeister II move relating these two diagrams of the unlink. Let $f_1,$ $f_2,$ and $f_3$ be representatives of the morphisms in (\ref{eqn:du12}), (\ref{eqn:u0u12}), and (\ref{eqn:du22}), respectively. Then the composition \[f_3\circ (id\otimes f_2)\circ f_1\] from a representative of $\cA(D_2)$ to a representative of $\cA(D_3)$ is a morphism which agrees on $E_2$ with the composition \begin{equation*}\label{eqn:maponkh}g_3\circ (id\otimes g_2)\circ g_1:\Kh(D_2)\to\Kh(D_3),\end{equation*} and this latter composition is equal to the isomorphism from $\Kh(D_2)$ to $\Kh(D_3)$ associated to the Reidemeister II move. It follows  that the morphism in (\ref{eqn:mor22}) exists.
\end{proof}

\begin{lemma}
\label{lem:R3}
Suppose $D'$ is obtained from $D$ by a Reidemeister III move. Then there exists a morphism \[\cA(D)\to\cA(D')\] which agrees on $E_2$ with the standard isomorphism from $\Kh(D)$ to $\Kh(D')$. \end{lemma}

\begin{proof}
Consider the  link diagrams  shown in Figure \ref{fig:R3}. The  arrow from $D=D_1$ to $D_2$ represents three $0$-handle attachments; the arrow from $D_2$ to $D_3$ represents a sequence consisting of three Reidemeister II moves; the arrow from $D_3$ to $D_4$ represents a Reidemeister III move;  the  arrow from $D_4$ to $D_5$ represents three $1$-handle attachments; and the arrow from $D_5$ to $D_6=D'$ represents a sequence of three Reidemeister II moves. The movie represented by these thin arrows  is equivalent to the movie from $D$ to $D'$ corresponding to the single Reidemeister III move indicated by the thick arrow. These two movies therefore induce the same map from $\Kh(D)$ to $\Kh(D')$. Thus, to prove Lemma \ref{lem:R2}, it suffices to prove that there exist morphisms 
\begin{align}
\label{eqn:mor123}\cA(D_1)\to\cA(D_2)\\
\label{eqn:mor223}\cA(D_2)\to\cA(D_3)\\
\label{eqn:mor323}\cA(D_3)\to\cA(D_4)\\
\label{eqn:mor423}\cA(D_4)\to\cA(D_5)\\
\label{eqn:mor523}\cA(D_5)\to\cA(D_6)
\end{align}
which agree on $E_2$ with the corresponding maps on Khovanov homology. Since the top left and bottom right arrows in Figure \ref{fig:R3} correspond to handle attachments, the morphisms  in (\ref{eqn:mor123}) and (\ref{eqn:mor423}) exist  by Proposition \ref{prop:02handle}. The top right  and bottom left arrows correspond to sequences of Reidemeister II moves, so the morphisms  in (\ref{eqn:mor223}) and (\ref{eqn:mor523}) exist by Lemma \ref{lem:R2}. It remains to show that the morphism in (\ref{eqn:mor323}) exists. 

\begin{figure}[ht]
\labellist
\hair 2pt
\pinlabel $D=D_1$ at 108 534
\pinlabel $D_2$ at 381 534
\pinlabel $D_3$ at 653 534
\pinlabel $D_4$ at 653 10
\pinlabel $D_5$ at 381 10
\pinlabel $D'=D_6$ at 105 10

\endlabellist
\centering
\includegraphics[height=7cm]{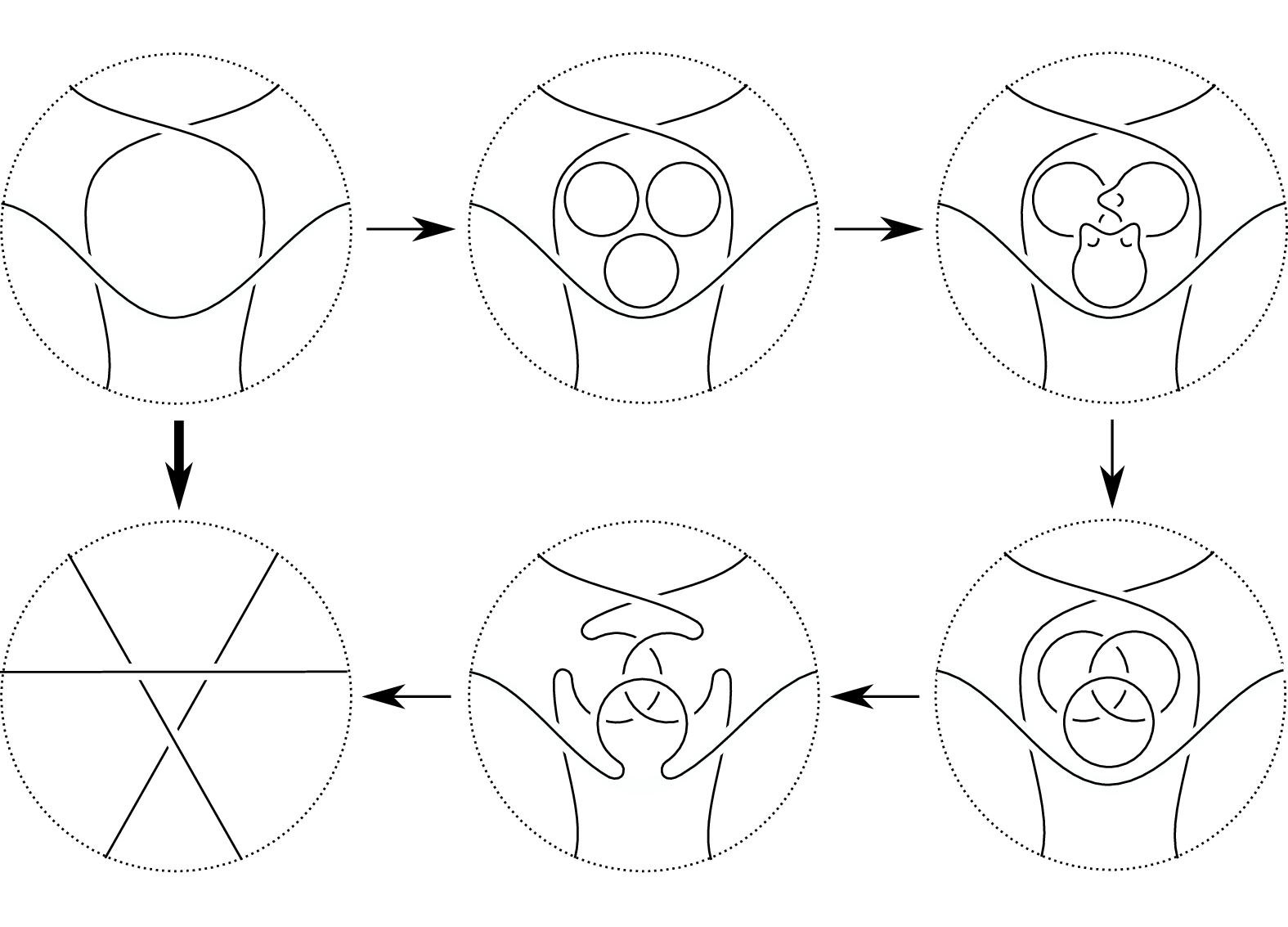}
\caption{ The diagrams $D=D_1,\dots,D_6=D'$. The movie indicated by the thin arrows is equivalent to the movie corresponding to the Reidemeister III move,  indicated by the thick arrow.}
\label{fig:R3}
\end{figure}
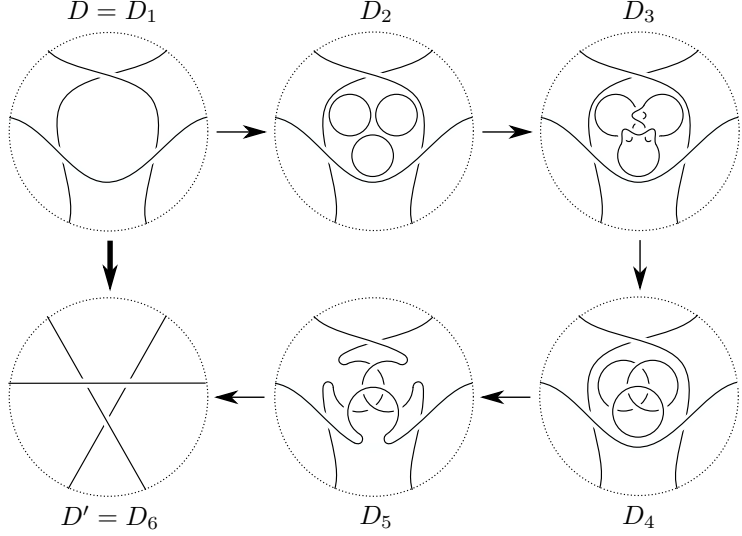

Let $U_a$ and $U_b$ be the  $6$-crossing diagrams of the $3$-component unlink  in  $D_3$ and $D_4$, so that $D_3=D_1\sqcup U_a$ and $D_4 = D_1\sqcup U_b$. By condition (3) in Definition \ref{def:kftheory}, there exist morphisms 
\begin{align}
\label{eqn:du13}&\cA(D_3)\to\cA(D_1)\otimes \cA(U_a)\\
\label{eqn:du23}&\cA(D_1)\otimes \cA(U_b)\to \cA(D_4)
\end{align}
which agree on $E_2$ with the standard isomorphisms  
\begin{align*}
g_1&:\Kh(D_3)\to\Kh(D_1)\otimes \Kh(U_a)\\
g_3&:\Kh(D_1)\otimes \Kh(U_b)\to \Kh(D_4).
\end{align*}
Condition (4) in Definition \ref{def:kftheory} says that 
\begin{align*}
E_\infty(\cA(U_i))&=E_2(\cA(U_i))\cong\Kh(U_i)
\end{align*}
for $i=a,b$, which implies as in the previous proof that  the quasi-isomorphism class $\cA(U_i)$ contains the trivial $\Kh(U_i)$-complex $(\Kh(U_i),id)$ for $i=a,b$. As before, it follows immediately that there exists a morphism 
\begin{equation}\label{eqn:u0u13}\cA(U_a)\to\cA(U_b)\end{equation} 
which agrees on $E_2$ with the standard isomorphism 
\begin{equation*}\label{eqn:ku0u13}g_2:\Kh(U_a)\to\Kh(U_b)\end{equation*}
associated to the Reidemeister III move relating these two diagrams of the unlink. Let $f_1,$ $f_2,$ and $f_3$ be representatives of the morphisms in (\ref{eqn:du13}), (\ref{eqn:u0u13}), and (\ref{eqn:du23}), respectively. Then the composition \[f_3\circ (id\otimes f_2)\circ f_1\] from a representative of $\cA(D_3)$ to a representative of $\cA(D_4)$ is a morphism which agrees on $E_2$ with the composition \begin{equation*}\label{eqn:maponkh}g_3\circ (id\otimes g_2)\circ g_1:\Kh(D_3)\to\Kh(D_4),\end{equation*} and this latter composition is equal to the isomorphism from $\Kh(D_3)$ to $\Kh(D_4)$ associated to the Reidemeister III move. It follows  that the morphism in (\ref{eqn:mor323}) exists.
\end{proof}

As mentioned in the previous section, the proof  that reduced Khovanov-Floer theories are functorial proceeds in a virtually identical manner; we leave it to the reader to fill in the details.

\section{Examples of Khovanov-Floer theories}
\label{sec:kfthys}

In the first four subsections, we verify that the  spectral sequence constructions of Kronheimer-Mrowka, Ozsv{\'a}th-Szab{\'o},  Szab{\'o}, and Lee define Khovanov-Floer theories, proving Theorems \ref{thm:kmoskf}, \ref{thm:szabokf}, and \ref{thm:leekf}.  All four verifications are formally very similar. Theorem \ref{thm:kfunct2} or its reduced analogue then imply that the associated spectral sequences define functors from $\link$ or $\link_\infty$ to $\spect$. As mentioned in the introduction, this also provides a new proof that Rasmussen's invariant is indeed a knot invariant.

In the fifth subsection, we describe some new deformations of the Khovanov complex which give rise to Khovanov-Floer theories, and potentially to new link type invariants (they of course give rise to functorial link type invariants, the question is whether these invariants are  different from Khovanov homology).


\subsection{Kronheimer and Mrowka's spectral sequence}
\label{ssec:kmss} 
Suppose $D\subset S^2:=\R^2\cup\{\infty\}$ is a diagram for an oriented link $L\subset S^3:=\R^3\cup\{\infty\}$, with crossings labeled $1,\dots,n$. For each $I\in\{0,1\}^n$, let $L_I\subset S^3$ be a link whose projection to $\R^2$ is equal to  $D_I$, and which agrees with $L$ outside of $n$ disjoint balls containing the ``crossings" of $L$. For every pair $I<_1 I'$ of immediate successors, there is a standard $1$-handle cobordism \[S_{I,I'}\subset S^3\times[0,1]\] from $L_I$ to $L_{I'}$ which is trivial outside the product of one of these balls with the interval. For any pair $I<_kJ$ of tuples differing in $k$ coordinates, choose a sequence $I = I_0<_1 I_1 <_1 \dots <_1 I_k = I'$ of immediate successors. Then the composition \[S_{I,J} = S_{I_{k-1},I_k}\circ\dots\circ S_{I_0,I_1}\] defines a cobordism \[S_{I,J}\subset S^3\times[0,1]\] from $L_I$ to $L_J$ which is independent of the sequence above, up to proper isotopy. 

Given some  auxiliary data $\adata$ (including  a host of metric and perturbation data) Kronheimer and Mrowka construct \cite{km3} a  chain complex $(C^\adata(D),d^\adata(D))$, where \[C^\adata(D) = \bigoplus_{I\in\{0,1\}^n} \CI(\bar L_I)\] and the differential $d^\adata(D)$ is a sum of maps \[d_{I,J}: \CI(\bar L_I)\to \CI(\bar L_J)\] over all pairs $I\leq J$ in $\{0,1\}^n$. Here, $\CI(\bar L_I)$ refers to the unreduced  singular instanton Floer chain group of $\bar L_I$ over $\F$.   The map $d_{I,I}$ is the    instanton Floer differential on $\CI(\bar L_I)$, defined, \emph{very roughly speaking}, by counting certain instantons on $S^3\times \R$ with singularities along $\bar L_I \times \R$. More generally, $d_{I,J}$ is defined by counting points in parametrized moduli spaces of instantons on $S^3\times \R$ with singularities along $S_{I,J}$, over a family of metrics and perturbations. We are abusing notation  here, of course, as the  vector spaces $\CI(\bar L_I)$ and  maps $d_{I,J}$ depend on $\adata$.

Kronheimer and Mrowka prove in \cite{km3} that the homology of this complex computes the unreduced singular instanton Floer homology of $\bar L$, as below.

\begin{theorem}{\em (Kronheimer-Mrowka \cite{km3})}
The homology $H_*(C^\adata(D),d^\adata(D))$ is isomorphic to $\HI(\bar L)$.
\end{theorem}

Note that the complex $(C^\adata(D),\partial^\adata(D))$  is a filtered complex with respect to the filtration coming from the  \emph{homological grading} defined by \[\mathbf{h}(x) = I_1+\dots +I_n-n_-\]  for   $x\in \CI(\bar L_I)$. Since $d_{I,I}$ is the instanton Floer differential, the $E_1$ page of the associated spectral sequence is given by \[E_1(C^\adata(D)) = \bigoplus_{I\in\{0,1\}^n}\HI(\bar L_I).\] Moreover,  the spectral sequence differential $d_1(C^\adata(D))$ is the   sum of the induced maps \[(d_{I,I'})_*:\HI(\bar L_I)\to \HI(\bar L_{I'})\] over all pairs $I<_1 I'$.

In \cite[Section 8]{km3}, Kronheimer and Mrowka establish  isomorphisms \[\Lambda^*V(D_I)\cong \HI(\bar L_I)\] which extend to an isomorphism of chain complexes \[(\cKh(D),d) \to (E_1(C^\adata(D)), d_1(C^\adata(D)))\] that  gives rise to an isomorphism \[q^\adata:\Kh(D)\to E_2(C^\adata(D)).\] 
Moreover, they show that for any two sets of data $\adata$ and $\adata'$, there exists a filtered chain map \[f:C^\adata(D)\to C^{\adata'}(D)\] such that \[E_2(f)=q^{\adata'}\circ (q^\adata)^{-1}.\] This is essentially the content of  \cite[Proposition 8.11]{km3} and the discussion immediately following it. In other words, Kronheimer and Mrowka's construction assigns to every   link diagram $D$ a quasi-isomorphism class of $\Kh(D)$-complexes, with respect to the homological grading on $\Kh(D)$. In fact, we claim the following.

\begin{proposition}
\label{prop:kmss}
Kronheimer-Mrowka's construction is a Khovanov-Floer theory.
\end{proposition}

\begin{proof}

Let $\cA(D)$ denote the quasi-isomorphism class of $\Kh(D)$-complexes assigned to $D$ in Kronheimer and Mrowka's construction. To prove the proposition, we   simply check   that $\cA$ satisfies conditions (1)-(4) of Definition \ref{def:kftheory}. 

For condition (1), a planar isotopy $\phi$ taking $D$ to $D'$ determines a canonical filtered (in fact, grading-preserving) chain isomorphism   \[\psi_\phi:C^\adata(D) \to C^{\adata'}(D'),\] where $\adata$ is the data pulled back from $\adata'$ via $\phi$. Furthermore, it is clear that $E_1(\psi_\phi)$ agrees with the standard map \[F_\phi:\CKh(D)\to\CKh(D')\] associated to this isotopy in Khovanov homology, with respect to the natural identifications of the various chain complexes. It follows that $\psi_\phi$ represents a morphism from $\cA(D)$ to $\cA(D')$ which agrees on $E_2$ with the map induced on Khovanov homology, as desired.

For condition (2), suppose $D'$ is obtained from $D$ via a diagrammatic 1-handle attachment. Then there is a diagram $\tilde D$ with one more crossing than $D$ and $D'$, such that $D$ is the $0$-resolution of $\tilde D$ at this new crossing $c$ and $D'$ is the $1$-resolution. For some choice of data $\tilde\adata$, we can realize the complex $C^{\tilde\adata}(\tilde D)$ as the mapping cone of a degree 0 filtered chain map \[T:C^{\adata}(D)\to C^{\adata'}(D'),\] where $\adata$ and $\adata'$ are appropriate restrictions of $\tilde\adata$. This map $T$ is given by the direct sum \[ T = \bigoplus_{I\leq J\in \{0,1\}^{n}} d_{I\times\{0\},J\times\{1\}},\] of   components of the differential $d^{\tilde\adata}(\tilde D)$. (We are thinking of $c$ as the $(n+1)^{\rm st}$ crossing of $\tilde D$.) Then \[E_1(T):E_1(C^{\adata}(D))\to E_1(C^{\adata'}(D'))\] is given by the direct sum of the maps \[(d_{I\times\{0\},I\times\{1\}})_*:\HI( \bar L_I)\to\HI( \bar L'_I)\] over all $I\in \{0,1\}^n$.  It follows from \cite[Proposition 8.11]{km3} that these maps  agree with the maps \[\Lambda^*V(D_I)\to \Lambda^*V(D'_I)\] associated to  the 1-handle addition, via the natural identifications 
 \begin{align*}
 \Lambda^*V(D_I)&\cong\HI(\bar L_I)\\
  \Lambda^*V(D'_I)&\cong\HI(\bar L'_I)
 \end{align*} described above. It follows that $E_1(T)$ agrees with the chain map \[\cKh(D)\to\cKh(D')\] associated to  the 1-handle attachment, and, hence, that  $T$ represents a morphism from $\cA(D)$ to $\cA(D')$ which agrees on $E_2$ with the map induced on  Khovanov homology, as desired.

For condition (3), it suffices to show that for some choices of data $\adata$, $\adata'$, $\adata''$, there is a degree 0 filtered chain map \begin{equation}\label{eqn:fcm}C^{\adata''}(D\sqcup D')\to C^{\adata}(D)\otimes C^{\adata'}(D')\end{equation} which agrees on $E_2$ with the standard isomorphism \[\Kh(D\sqcup D') \to \Kh(D)\otimes \Kh(D').\]  From the construction in \cite[Subsection 5.5]{km3},  there is an excision cobordism  which gives rise to a quasi-isomorphism \begin{equation}\label{eqn:qi}\CI(\bar L_I \sqcup \bar L'_J) \to \CI(\bar L_I)\otimes \CI(\bar L'_J)\end{equation} for some such data and each pair $I,J\in\{0,1\}^n$. Moreover,  since the induced isomorphism \[\HI(\bar L_I \sqcup \bar L'_J) \to \HI(\bar L_I)\otimes \HI(\bar L'_J)\]  is natural with respect to ``split" cobordisms (see \cite[Corollary 5.9]{km3}),  it follows  that this  isomorphism agrees with the isomorphism 
 \[\Lambda^*V(D_I\sqcup D_J) \to \Lambda^* V(D_I)\otimes \Lambda^* V(D_J)\] modulo the relevant natural identifications. The proof is then complete as long as one can show that the chain maps in (\ref{eqn:qi}) arise as the degree 0 components of a degree 0 filtered chain map as in (\ref{eqn:fcm}). Although we do not give details, this can be arranged, defining the higher degree components of the chain map by counting instantons on the excision cobordism over higher dimensional families of metrics and perturbations, mimicking the arguments in \cite[Section 6]{km3}. 
 
 
For condition (4), suppose $D$ is a diagram of the unlink. Then its Khovanov homology is supported in homological degree 0. Hence, the spectral sequence collapses at the $E_2$ page. In particular, $E_2(\cA(D)) = E_\infty(\cA(D))$, as desired.
\end{proof}


\subsection{Ozsv{\'a}th and Szab{\'o}'s spectral sequence}

Suppose $D$, $L$, and the $L_I$ are exactly as in the previous subsection, except that they are based at $\infty$. Let $a_j$ be an arc in a local neighborhood of the $j$th crossing of $D$ as shown in Figure \ref{fig:arc}, and let $b_j$ be a lift of $a_j$ to an arc in $S^3$ with endpoints on $L$. The arc $b_j$ lifts to a closed curve $\beta_j\subset-\Sigma(L)$, where $\Sigma(L)$ is the double branched cover of $S^3$ branched along $L$. There is a natural framing on the link \[\mathbb{L}=\beta_1\cup\dots\cup \beta_n\subset -\Sigma(L)\] such that $-\Sigma(L_I)$ is obtained by performing $I_j$-surgery on $\beta_j$ for each $j=1,\dots,n$, for all $I\in\{0,1\}^n$.

 \begin{figure}[!htbp]
 \labellist 
\small\hair 2pt  
\pinlabel $a_j$ at 11 27 
\endlabellist 
\begin{center}
\includegraphics[width=2.1cm]{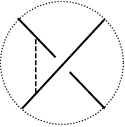}
\caption{\quad The arc $a_j$ near the $j$th crossing, shown as a dashed segment.}
\label{fig:arc}
\end{center}
\end{figure} 

Given some auxiliary data $\adata$ (including a pointed Heegaard multi-diagram  \emph{subordinate} to the framed link $\mathbb{L}$ and a host of complex-analytic data), Ozsv{\'a}th and Szab{\'o} construct \cite{osz12} a  chain complex $(C^\adata(D),d^\adata(D))$, where \[C^\adata(D) = \bigoplus_{I\in\{0,1\}^n} \CF(-\Sigma(L_I))\] and the differential $d^\adata(D)$ is a sum of maps \[d_{I,J}:\CF(-\Sigma(L_I))\to\CF(-\Sigma(L_{J}))\] over all pairs $I\leq J$ in $\{0,1\}^n$. Here, $\CF(-\Sigma(L_I))$ refers to the Heegaard Floer chain group of $-\Sigma(L_I)$. The map $d_{I,I}$ is the usual Heegaard Floer differential on $\CF(-\Sigma(L_I))$, defined by counting pseudo-holomorphic disks in the symmetric product of a Riemann surface. More generally, $d_{I,J}$ is defined by counting pseudo-holomorphic polygons. Again, we are abusing notation here, as the  vector spaces $\CF(-\Sigma(L_I))$ and  maps $d_{I,J}$ depend on $\adata$. 

Ozsv{\'a}th and Szab{\'o} prove in \cite{osz12} that the homology of this complex computes the Heegaard  Floer homology of $-\Sigma(L)$; that is:

\begin{theorem}
The homology $H_*(C^\adata(D),d^\adata(D))$ is isomorphic to $\HF(-\Sigma(L))$.
\end{theorem}

As in the previous subsection, this complex $(C^\adata(D),\partial^\adata(D))$  is filtered with respect to the obvious homological grading. Since $d_{I,I}$ is the Heegaard Floer differential, the $E_1$ page of the associated spectral sequence is given by \[E_1(C^\adata(D)) = \bigoplus_{I\in\{0,1\}^n}\HF(-\Sigma(L_I)).\] Moreover,  the spectral sequence differential $d_1(C^\adata(D))$ is the   sum of the induced maps \[(d_{I,I'})_*:\HF(-\Sigma(L_I))\to\HF(-\Sigma(L_{I'}))\] over all pairs $I<_1 I'$.

Below, we argue that Ozsv{\'a}th and Szab{\'o}'s construction  assigns to $D$ a quasi-isomorphism class of $\Khr(D)$-complexes.

In general, the Heegaard Floer homology of a 3-manifold $Y$ admits an action by $\Lambda^*H_1(Y)$. For each $I\in\{0,1\}^n$,  the Floer homology $\HF(-\Sigma(L_I))$ is a  free module over \[\Lambda^*H_1(-\Sigma(L_I))\] of rank one, generated by the  unique element  in the top Maslov grading. In particular, there is a  canonical identification \begin{equation}\label{eqn:canonic}\HF(-\Sigma(L_I))\cong\Lambda^*H_1(-\Sigma(L_I)).\end{equation} Suppose $x$ is the component of $D_I$ containing the basepoint $\infty$. Given any other component $y$, let $\eta_{x,y}$ be an arc with endpoints on $L_I$ which projects to an arc from $x$ to $y$. The map \[V(D_I)/(x)\to H_1(-\Sigma(L_I))\] which sends a component $y$ to the homology class of the lift of $\eta_{x,y}$ to the branched double cover clearly gives rise to an isomorphism \[\Lambda^*(V(D_I)/(x)) \to \HF(-\Sigma(L_I))\] via the identification in (\ref{eqn:canonic}). Moreover, Ozsv{\'a}th and Szab{\'o} show that the direct sum of these isomorphisms  gives rise to an isomorphism of chain complexes \[(\cKhr(D),d) \to (E_1(C^\adata(D)), d_1(C^\adata(D))).\] This isomorphism then gives rise to an isomorphism \[q^\adata:\Khr(D)\to E_2(C^\adata(D)).\] It follows from the work in \cite{bald7,lrob3} and  naturality properties of the $\Lambda^*H_1$-action that for any two sets of data $\adata$ and $\adata'$, there exists a filtered chain map \[f:C^\adata(D)\to C^{\adata'}(D)\] such that \[E_2(f)=q^{\adata'}\circ (q^\adata)^{-1}.\] This shows that Ozsv{\'a}th and Szab{\'o}'s construction assigns to every based  link diagram $D$ a quasi-isomorphism class of $\Khr(D)$-complexes, with respect to the homological grading on $\Khr(D)$. In fact, we claim the following.

\begin{proposition}
\label{prop:osss}
Ozsv{\'a}th and Szab{\'o}'s  construction is a reduced Khovanov-Floer theory.
\end{proposition}

\begin{proof}
Let $\cA(D)$ denote the quasi-isomorphism class of $\Khr(D)$-complexes assigned to $D$ in Ozsv{\'a}th and Szab{\'o}'s construction. We   verify  below that $\cA$ satisfies the reduced analogues of conditions (1)-(4) of Definition \ref{def:kftheory}. 

For condition (1), a planar isotopy $\phi$ taking $D$ to $D'$ determines a canonical filtered (in fact, grading-preserving) chain isomorphism   \[\psi_\phi:C^\adata(D) \to C^{\adata'}(D'),\] where $\adata$ is the data pulled back from $\adata'$ via $\phi$, just as in the instanton case. Furthermore, it is clear that $E_1(\psi_\phi)$ agrees with the standard map \[F_\phi:\CKhr(D)\to\CKhr(D')\] associated to this isotopy in reduced Khovanov homology, with respect to the natural identifications of the various chain complexes. It follows that $\psi_\phi$ represents a morphism from $\cA(D)$ to $\cA(D')$ which agrees on $E_2$ with the map induced on reduced Khovanov homology, as desired.

For condition (2), Suppose $D'$ is obtained from $D$ via a 1-handle attachment. Let $\tilde D$ be a diagram with one more crossing than $D$ and $D'$, such that $D$ is the $0$-resolution of $\tilde D$ at this  crossing  and $D'$ is the $1$-resolution, as in the proof of Proposition \ref{prop:kmss}. Following that proof,  we can realize the complex $C^{\tilde\adata}(\tilde D)$ as the mapping cone of a degree 0 filtered chain map \[T:C^{\adata}(D)\to C^{\adata'}(D'),\] for some choice of data $\tilde \adata$ and the appropriate restrictions  $\adata$ and $\adata'$. As before, $T$ is given by the direct sum \[ T = \bigoplus_{I\leq J\in \{0,1\}^{n}} d_{I\times\{0\},J\times\{1\}},\] of   components of the differential $d^{\tilde\adata}(\tilde D)$, and \[E_1(T):E_1(C^{\adata}(D))\to E_1(C^{\adata'}(D'))\] is  the direct sum of the maps \[(d_{I\times\{0\},I\times\{1\}})_*:\HF(-\Sigma(L_{I}))\to\HF(-\Sigma(L'_{I}))\] over all $I\in \{0,1\}^n$.  It is easy to see that these maps  agree with the maps \[\Lambda^*(V(D_I)/(x))\to \Lambda^*(V(D'_I)/(x'))\] associated to  the 1-handle attachment, via the natural identifications 
 \begin{align*}
 \Lambda^*(V(D_I)/(x))&\cong\HF(-\Sigma(L_{I}))\\
  \Lambda^*(V(D'_I)/(x'))&\cong\HF(-\Sigma(L'_{I})),
 \end{align*} where $x$ and $x'$ are the components of $D_I$ and $D'_I$ containing the basepoint $\infty$. It follows that $E_1(T)$ agrees with the chain map \[\cKhr(D)\to\cKhr(D')\] associated to  the 1-handle attachment, and, hence, that  $T$ represents a morphism from $\cA(D)$ to $\cA(D')$ which agrees on $E_2$ with the map induced on reduced Khovanov homology, as desired.

For condition (3), it suffices as in the instanton Floer case to show that for some sets of data $\adata$, $\adata'$, $\adata''$, there is a degree 0 filtered chain map \[C^{\adata''}(D\sqcup D')\to C^{\adata}(D)\otimes C^{\adata'}(D'\sqcup U_\infty)\] which agrees on $E_2$ with the standard isomorphism \[\Khr(D\sqcup D') \to \Khr(D)\otimes \Khr(D'\sqcup U_{\infty}),\] where $D$ and $D$' are disjoint diagrams with $D$ containing $\infty$, as at the end of Subsection \ref{ssec:kh}. But, given the Heegaard multi-diagrams encoded by $\adata$ and $\adata'$, one can simply take an appropriate connected sum to produce a multi-diagram giving rise to a complex $C^{\adata''}(D\sqcup D')$ which is   isomorphic to the tensor product \[C^{\adata}(D)\otimes C^{\adata'}(D'\sqcup U_{\infty})\] by an isomorphism which agrees on $E_2$ with the map on reduced Khovanov homology (see \cite[Lemma 3.4]{bald7}).
 
 
For condition (4), suppose $D$ is a diagram of the unlink. Then its reduced Khovanov homology is supported in homological degree 0. Hence, the spectral sequence collapses at the $E_2$ page. In particular, $E_2(\cA(D)) = E_\infty(\cA(D))$, as desired.
 \end{proof}

\subsection{Szab{\'o}'s geometric spectral sequence}
\label{ssec:szabogss}
Suppose $D$ is a link diagram as in Subsection \ref{ssec:kmss}. Given  auxiliary data $\adata$ consisting of  a \emph{decoration} of $D$, Szab{\'o} defines in \cite{szabo} a chain complex $(C^\adata(D),d^\adata(D))$, where \[C^\adata(D) = \bigoplus_{I\in\{0,1\}^n} \cKh(D_I)\] and the differential $d^\adata(D)$ is a sum of maps \[d_{I,J}: \Lambda^*V( D_I)\to \Lambda^*V( D_J)\] over all pairs $I\leq J$ in $\{0,1\}^n$. The maps $d_{I,I}$ are identically zero. For $I<_1 I'$, the map $d_{I,I'}$ is the usual merge or split map used to define the Khovanov differential; in particular, it does not depend on the decoration $\adata$.  The  maps $d_{I,J}$ are also defined combinatorially, but do depend on $\adata$. It is an interesting question what the homology of this complex computes. Szab{\'o} conjectures the following in \cite{szabo}.

\begin{conjecture}
 The homology $H_*(C^\adata(D),d^\adata(D))$ is isomorphic to $\oplus^2\HF(-\Sigma(L))$.
\end{conjecture}

The complex $(C^\adata(D),d^\adata(D))$ is obviously filtered with respect to the homological grading. By construction, \[(\cKh(D),d) = (E_1(C^\adata(D)), d_1(C^\adata(D))),\] so that \[ \Kh(D)=E_2(C^\adata(D))\] on the nose. We may therefore define each $q^\adata$ to be the identity map. Szab{\'o} shows that for any two sets of data $\adata$ and $\adata'$, there exists a filtered chain map \[f:C^\adata(D)\to C^{\adata'}(D)\] which is equal to the identity map on $E_2$. In particular, Szab{\'o}'s construction assigns to every   link diagram $D$ a quasi-isomorphism class of $\Kh(D)$-complexes, with respect to the homological grading on $\Kh(D)$. As before, we claim the following.

\begin{proposition}
\label{prop:sss}
Szab{\'o}'s construction is a Khovanov-Floer theory.
\end{proposition}

\begin{proof}
Let $\cA(D)$ denote the quasi-isomorphism class of $\Kh(D)$-complexes assigned to $D$ in Szab{\'o}'s construction. The proof of this proposition is again just a verification that $\cA$ satisfies  conditions (1)-(4) of Definition \ref{def:kftheory}.

For condition (1), we proceed exactly as in the previous two subsections.

For condition (2), we also proceed as in those subsections. Let $D$, $D'$, and $\tilde D$ be diagrams described previously. We may choose a decoration $\tilde \adata$ for $\tilde D$ which restricts to decorations $\adata$ and $\adata'$ for $D$ and $D'$. It follows from the \emph{Extension Formula} in \cite[Definition 2.5]{szabo} that we can realize the complex $C^{\tilde\adata}(\tilde D)$ as the mapping cone of the degree 0 filtered chain map \[T:C^{\adata}(D)\to C^{\adata'}(D')\]  given by the direct sum \[ T = \bigoplus_{I\leq J\in \{0,1\}^{n}} d_{I\times\{0\},J\times\{1\}},\] of   components of the differential $d^{\tilde\adata}(\tilde D)$. Then \[E_1(T):E_1(C^{\adata}(D))\to E_1(C^{\adata'}(D'))\] is given by the direct sum of the maps \[d_{I\times\{0\},I\times\{1\}}:\Lambda^*V(  D_I)\to\Lambda^*V(  D'_I)\] over all $I\in \{0,1\}^n$.  But  these   are precisely   the maps associated to  the 1-handle attachment. Thus, $E_1(T)$ agrees with the chain map \[\cKh(D)\to\cKh(D')\] associated to  the 1-handle attachment. It follows that $T$ represents a morphism from $\cA(D)$ to $\cA(D')$ which agrees on $E_2$ with the map induced on  Khovanov homology, as desired.

For condition (3), suppose $\adata$ and $\adata'$ are decorations for  $D$ and $D'$, and let $\adata''$ be the corresponding decoration for $D\sqcup D'$. Then  the \emph{Disconnected Rule}  \cite[Definition 2.7]{szabo} implies that 
\[C^{\adata''}(D\sqcup D')=C^{\adata}(D)\otimes C^{\adata'}(D')\] as complexes. It follows immediately that $\cA$ satisfies condition (3).

For condition (4), suppose $D$ is a diagram of the unlink. Then its Khovanov homology is supported in homological degree 0. Hence, the spectral sequence collapses at the $E_2$ page. In particular, $E_2(\cA(D)) = E_\infty(\cA(D))$, as desired.
\end{proof}

\subsection{Lee's spectral sequence}
\label{ssec:lee}
Let $D$ be a link diagram as in the previous subsections. 
In \cite{Lee}, Lee defined a perturbation of the Khovanov complex of $D$ which, over $\mathbb{Q}$, gives rise to a spectral sequence with $E_2$ page the Khovanov homology $\Kh(D)$ and abutting to $(\Q\oplus \Q)^{k},$ where $k$ is the number of components of $D$. When $D$ is a knot diagram,  Rasmussen's numerical invariant $s_{\Q}$ \cite{ras3} mentioned in the introduction may be  defined as the average of the quantum gradings on  the two  summands of the $E_{\infty}$ page of this spectral sequence. This invariant  defines a  homomorphism from the smooth concordance group to $\mathbb{Z}$, and provides a lower bound on the smooth slice genus. 

In \cite{bncob}, Bar-Natan defined a version of Lee's construction for coefficients in $\F$, with the corresponding properties as above. Roughly speaking, Bar-Natan's theory  is built from the $(1+1)$-dimensional TQFT associated with the Frobenius algebra $\F[x]/(x^2+x)$ while Lee's theory corresponds to the Frobenius algebra $\Q[x]/(x^2 - 1)$.

Bar-Natan's construction assigns to  $D$ a chain complex $(C(D), d^{BN})$, where \[C(D) =\bigoplus_{I\in\{0,1\}^n} \cKh(D_I)\] and the differential $d^{BN}$  is a sum of maps \[d^{BN}_{I,I'}:\Lambda^*V(D_I)\to \Lambda^*V(D_{I'})\] over all pairs $I<_1 I' \in \{0,1\}^n$. Here, \[d^{BN}_{I,I'} = d_{I,I'} + d'_{I,I'},\] where $d_{I,I'}$ is the standard merge or split map used to define the Khovanov differential, and $d'_{I,I'}$ is a map which raises the quantum grading by $2$. 
Turner proves the following in \cite{tur}.

\begin{theorem}
The homology $H_*(C(D), d^{BN})\cong (\F\oplus \F)^{k}$, where $k$ is the number of components of $D$.
\end{theorem}

Note that $(C(D), d^{BN})$ is a filtered complex with respect to the quantum grading. Furthermore, \[\Kh(D)=E_1(C(D))=E_2(C(D))\] for the associated spectral sequence. It is thus clear that Bar-Natan's construction assigns to a link diagram $D$ a quasi-isomorphism class of $\Kh(D)$-complexes, with respect to the quantum grading on $\Kh(D)$. Moreover, we have the following.

\begin{proposition}
Bar-Natan's construction is a Khovanov-Floer theory.
\end{proposition}

\begin{proof} The proof proceeds almost exactly   as in the previous subsections, but is even easier; we omit it here.
\end{proof}

As mentioned above, if $D$ is a knot diagram, then Rasmussen's invariant $s_\F$ can be defined as the average of quantum gradings of the two summands of \[E_{\infty}(C(D))\cong \F\oplus \F.\]  A priori, this average depends on the diagram $D$. The fact that Bar-Natan's spectral sequence is functorial provides an independent proof that this average is, in fact, a knot invariant.

\subsection{New knot invariants}
\label{subsec:new_knot_invariants}
In addition to gathering known spectral sequences from Khovanov homology under the umbrella of our Khovanov-Floer formalism, there is an opportunity to search for combinatorial perturbations of the Khovanov differential which give rise to Khovanov-Floer theories.  Our main result shows that any such perturbation gives a spectral sequence which is a functorial knot invariant.  Szab{\'o}'s geometric spectral sequence \cite{szabo} and Lee's deformation \cite{Lee} provide examples in which the resulting spectral sequence may be non-trivial.

In this subsection we give three examples of combinatorial perturbations in order to stimulate further work in classifying such perturbations and in computing their spectral sequences.  We do not know if, for example, the spectral sequence of any perturbation that we give here necessarily collapses at the $E_2$ page for all links.

Suppose that $I,J \in \{0,1\}^n$ such that  $I < _k J$, and choose a sequence of immediate successors
\[ I = I_0 <_1 I_1 <_1 I_2 <_1 \cdots <_1 I_k = J {\rm .} \]
For a planar diagram $D$ with crossings  $1,\dots,n$, this sequence defines a map
\[ d_{I,J} = d_{I_{k-1}, I_k} \circ \cdots \circ d_{I_0, I_1} : \Lambda^* V(D_I) \rightarrow \Lambda^* V(D_J) {\rm .} \]
Note that this map does not depend on the choice of sequence since 2-dimensional faces in the Khovanov cube commute.

Now we define the endomorphism
\[ d_k = \bigoplus_{I <_k J} d_{I,J} : \cKh(D) \rightarrow \cKh(D) \]
for each $k \geq 1$.  Note that each $d_k$ preserves the quantum grading and shifts the homological grading by $k$, and that $d_1$ is the Khovanov differential.
Finally, for any sequence $\underline{\mathbf{a}} = (a_1, a_2, a_3, a_4, \ldots)$ where $a_i \in \F$ for all $i \geq 1$ and $a_1 = 1$ we define the endomorphism
\[d_{\underline{\mathbf{a}}} = \bigoplus_{k \geq 1} a_k d_k : \cKh(D) \rightarrow \cKh(D) {\rm .}\]

We shall check that $d_{\underline{\mathbf{a}}}^2 = 0$ and leave it as an (easy) exercise for the reader to verify that this defines a Khovanov-Floer theory with a homological filtration and a quantum grading.

The key ingredient in the check is that $d_{I,K} = d_{J,K} \circ d_{I,J}$ for any $I \leq J \leq K$.  For convenience if $k$ is an odd integer then set $a_{k/2} = 0$ and set the binomial coefficient
\[ \left(
\begin{array}{c}
k\\
k/2\\
\end{array}
\right) = 0 {\rm .} \]
We have
\begin{eqnarray*}
	d_{\underline{\mathbf{a}}}^2 &=& \bigoplus_{i,j \geq 1} a_{j}d_{j} \circ a_{i}d_{i} = \bigoplus_{i,j \geq 1}(a_ja_i) d_{j} \circ d_{i} \\
	&=& \bigoplus_{\stackrel{I <_i J <_j K}{i,j\geq 1}} (a_j a_i) d_{J,K} \circ d_{I,J} = \bigoplus_{\stackrel{I <_i J <_j K}{i,j \geq 1}} (a_j a_i) d_{I,K} \\
	&=& \bigoplus_{\stackrel{I <_k K}{k \geq 2,k-1 \geq j \geq 1}} (a_{j} a_{k-j}) \left(
	\begin{array}{c}
		k\\
		j\\
	\end{array}
	\right) d_{I,K} \\
	&=& \bigoplus_{\stackrel{I <_k K}{k \geq 2, k/2 > j \geq 1}}  \left(2 \left(
	\begin{array}{c}
		k\\
		j\\
	\end{array}
	\right)(a_{j} a_{k-j}) + \left(
	\begin{array}{c}
		k\\
		k/2\\
	\end{array}
	\right)(a_{k/2} a_{k/2}) \right) d_{I,K} = 0 {\rm .}
\end{eqnarray*}

This concludes the first example.  For the second example, consider the same set-up of a planar diagram $D$ with crossings $1,\dots,n$.  Now look for all pairs $(I,J) \in \{0,1\}^n$ such that $I <_2 J$ and such that if $I <_1 K <_1 J$ then the movie represented by the sequence \[D_I, D_K, D_J\] consisting of two $1$-handle attachments describes a cobordism which is the union of a twice-punctured torus with some annuli.  We call such a pair $(I,J)$ a \emph{ladybug configuration} and write the set of all ladybug configurations as $\mathbf{L}$.

For $(I,J) \in \mathbf{L}$ we wish to define a map
\[ d'_{I,J} : \Lambda^* V(D_I) \rightarrow \Lambda^* V(D_J) {\rm .}\]
To do this, we first identify $D_I$ with $D_J$ (and so $\Lambda^* V(D_I)$ with $\Lambda^* V(D_J)$) by identifying circles which are part of the same connected component of the cobordism.  Then the map $d'_{I,J}$ is defined to be wedging with the generator of $\Lambda^* V(D_I)$ corresponding to a boundary component of the genus 1 part of the cobordism.

Now define the endomorphism
\[ d_L = d \oplus \bigoplus_{(I,J) \in \mathbf{L}} d'_{I,J} : \cKh(D) \rightarrow \cKh(D) {\rm ,} \]
where $d$ is the Khovanov differential.  Note that, as in the case of $d_{\underline{\mathbf{a}}}$, we have that $d_L$ preserves the quantum grading.

Again, once it is verified that $d_L^2 = 0$, it is an easy exercise to see that this defines a Khovanov-Floer theory.  The check that $d_L^2 = 0$ is combinatorial.  Explicitly, for any $I <_3 J$ we need to verify that
\[ \bigoplus_{I <_1 K <_2 J} d'_{K,J} \circ d_{I,K} \oplus \bigoplus_{I <_2 K <_1 J} d_{K,J} \circ d'_{I,K} : \Lambda^* V(D_I) \rightarrow \Lambda^* V(D_J) = 0 {\rm ,} \]
and for any $I <_4 J$ we need to verify that
\[ \bigoplus_{I <_2 K <_2 J} d'_{K,J} \circ d'_{I,K} : \Lambda^* V(D_I) \rightarrow \Lambda^* V(D_J) = 0 {\rm .} \]
Both checks may be made along the same lines as the checks in Szabo's \cite{szabo}, although this case is easier since there is no auxiliary data of a \emph{decoration}.  The first check should be carried out for each 3-dimensional configuration, and the second for each 4-dimensional configuration.  We leave these checks for the reader to verify.

Finally we very briefly give an example that makes use of a quantum rather than a homological filtration.  The idea can be summarized simply as replacing the ``saddle" differential by the sum of a saddle and a dotted saddle (in the sense of Bar-Natan \cite{bncob}).  Then all differentials raise the homological grading by $1$, while respecting a quantum filtration.

Explicitly, if $I <_1 J$ we define components of the deformed differential as
\[ d'_{I,J} = d_{I,J} + d_{I,J} \wedge x \]
where $d_{I,J}$ is the Khovanov differential, and where by $\wedge x$ we mean post composition by wedging with a generator $x$ corresponding to one of the (possibly two) circles of the resolution $J$ in the boundary of the pair of pants cobordism component.  This is independent of the choice of $x$ (as the ``dotting" formalism above suggests).

\begin{remark}
The first deformation above in the case $\underline{\mathbf{a}} = (1,1,1,\ldots)$ was studied independently by Juh{\'a}sz and Marengon.  In \cite[Section 6]{juhaszmarengon}, they also show that the isomorphism class of the resulting spectral sequence is a link type invariant.
\end{remark}

\bibliographystyle{hplain}
\bibliography{References}
\end{document}